\documentclass[12pt,preprint]{amsart}
\usepackage{graphicx}
\usepackage{amsmath}
\usepackage{amsthm}
\usepackage{amsfonts}
\usepackage{amssymb}
\usepackage[usenames,dvipsnames]{color}
\usepackage{verbatim}

\theoremstyle{plain}
\newtheorem{thm}{Theorem}[section]
\newtheorem{lemma}[thm]{Lemma}

\newtheorem{prop}[thm]{Proposition}

\newtheorem*{thm*}{Theorem}

\theoremstyle{remark}
\newtheorem{rmk}[thm]{Remark}

\theoremstyle{definition}
\newtheorem{defn}[thm]{Definition}

\numberwithin{equation}{section}

\def\N{\mathbb{N}}
\def\P{\mathbb{P}}
\def\R{\mathbb{R}}
\def\Z{\mathbb{Z}}

\def\cA{\mathcal{A}}
\def\cC{\mathcal{C}}
\def\cF{\mathcal{F}}

\def\1{\mathbf{1}}

\def\PP{\mathbb{P}_{n,\alpha}}
\def\EE{\mathbb{E}_{n,\alpha}}

\newcommand{\rArea}[1]{A(#1)}

\DeclareMathOperator{\interior}{int}
\DeclareMathOperator{\Line}{Line}
\DeclareMathOperator{\Per}{Per}

\DeclareMathOperator{\Var}{Var_{n,\alpha}}

\newenvironment{PfofEmptinessThm}[1]
{\par\vskip2\parsep\noindent{\sc Proof of Theorem\ \ref{Thm:EmptinessIntro}.}}{{\hfill
$\Box$}
\par\vskip2\parsep}

\newenvironment{PfofEntropyThm}[1]
{\par\vskip2\parsep\noindent{\sc Proof of Theorems\ \ref{Thm:EntropyIntro} and \ref{Thm:PeriodicEntropyIntro}.}}{{\hfill
$\Box$}
\par\vskip2\parsep}

\newenvironment{PfofExoticThm}[1]
{\par\vskip2\parsep\noindent{\sc Proof of Theorem\ \ref{Thm:ExoticBehavior}.}}{{\hfill
$\Box$}
\par\vskip2\parsep}

\title{Random $\Z^d$-shifts of finite type}

\begin{document}

\author{Kevin McGoff}
\address{Duke University}
\email{mcgoff@math.duke.edu}

\author{Ronnie Pavlov}
\address{University of Denver}
\email{rpavlov@du.edu}



\begin{abstract}
In this work we consider an ensemble of random $\Z^d$-shifts of finite type ($\Z^d$-SFTs) and prove several results concerning the behavior of typical systems with respect to emptiness, entropy, and periodic points. These results generalize statements made in \cite{McGoff} regarding the case $d=1$. 

Let $\mathcal{A}$ be a finite set, and let $d \geq 1$. For $n$ in $\N$ and $\alpha$ in $[0,1]$, define a random subset $\omega$ of $\mathcal{A}^{[1,n]^d}$ by independently including each pattern in $\mathcal{A}^{[1,n]^d}$ with probability $\alpha$. Let $X_{\omega}$ be the (random) $\Z^d$-SFT built from the set $\omega$. For each $\alpha \in [0,1]$ and $n$ tending to infinity, we compute the limit of the probability that $X_{\omega}$ is empty, as well as the limiting distribution of entropy of $X_{\omega}$. Furthermore, we show that the probability of obtaining a nonempty system without periodic points tends to zero.

For $d>1$, the class of $\Z^d$-SFTs is known to contain strikingly different behavior than is possible within the class of $\Z$-SFTs. Nonetheless, the results of this work suggest a new heuristic: typical $\Z^d$-SFTs have similar properties to their $\Z$-SFT counterparts.

\end{abstract}

\maketitle


\section{Introduction}

A shift of finite type (SFT) is a dynamical system defined by finitely many local rules. SFTs have been studied for their own sake \cite{Kitchens,LindMarcus}, for their connections to other dynamical systems \cite{Bowen,DGS,Keane}, and for their connections to the thermodynamic formalism and equilibrium states in statistical mechanics \cite{Bowen,Ruelle}. We consider SFTs defined on a lattice $\Z^d$, with $d \geq 1$. As these systems may be defined by finitely many combinatorial constraints, they naturally lend themselves to probabilistic models. In fact, random $\Z^d$-SFTs may be viewed as a class of random constraint satisfaction problems (for other examples of random constraint satisfaction problems, see \cite{AM2013,ACR2011,ADP2005,AP2004,Friedgut1999,KMRSZ2007}). 

In previous work \cite{McGoff}, the first author began a line of investigation aimed at describing likely properties of $\Z$-SFTs that have been selected according to a natural probability distribution, there called random SFTs. The main goal of the present work is to investigate some likely properties of random $\Z^d$-SFTs for arbitrary $d \geq 1$. 

For $d > 1$, the class of $\Z^d$-SFTs contains behavior that cannot appear in $\Z$-SFTs. Consider the following examples of this phenomenon. 
\begin{itemize}
\item There is an algorithm that decides in finite time whether a given $\Z$-SFT is empty, whereas no such algorithm exists for $\Z^d$-SFTs when $d > 1$ \cite{Berger1966}. 
\item Every nonempty $\Z$-SFT contains a finite orbit, but for each $d > 1$, there exist nonempty $\Z^d$-SFTs that contain no finite orbits \cite{Berger1966}.
\item The set of real numbers that appear as the entropy of some $\Z$-SFT has an algebraic characterization \cite{LindRealization}, whereas the corresponding set for $\Z^d$-SFTs  with $d >1$ has a computational characterization \cite{HM}.
\end{itemize}
For further illustrations of this phenomenon, see \cite{BPS,BurtonSteif,Hochman,QuasSahin2003}. In light of the presence of such exotic behavior within the class of $\Z^d$-SFTs for $d >1$, one may ask, ``how typical is this behavior?'' In principle, the setting of random $\Z^d$-SFTs allows one to answer such questions in a probabilistic sense. From this point of view, a property of $\Z^d$-SFTs may be considered typical if it holds with high probability. Indeed, our main results establish typical properties of $\Z^d$-SFTs in this sense.

We construct random $\Z^d$-SFTs as follows. Consider a finite set $\cA$ and a natural number $d$. Let $F_n = [1,n]^d$ be the hypercube with side length $n$ in $\Z^d$. Consider $\cA^{F_n}$, the set of all patterns on $F_n$, and let $\Omega_n$ be the power set of $\cA^{F_n}$. For $n \geq 1$ and $\alpha \in [0,1]$, let $\PP$ denote the probability measure on $\Omega_n$ given by including elements of $\cA^{F_n}$ independently with probability $\alpha$ (and excluding them with probability $1-\alpha$). The measure $\PP$ depends on $\mathcal{A}$ and $d$, but we suppress this dependence in our notation. For a subset $\omega$ of $\cA^{F_n}$, let $X_{\omega}$ be the $\Z^d$-SFT built from $\omega$:
\begin{equation*}
X_{\omega} = \{ x \in \cA^{\Z^d} : \forall p \in \Z^d, \, x|_{F_n+p} \in \omega \}.
\end{equation*}
We view the elements of $\cA^{F_n} \setminus \omega$ as ``forbidden patterns,'' which place constraints on the configurations allowed in $X_{\omega}$.
When $\omega$ is chosen at random according to $\PP$, we view $X_{\omega}$ as a random $\Z^d$-SFT.  

Our first main result concerns the probability that $X_{\omega}$ is the empty set. To state the result, we define the following zeta function as a formal power series:
\begin{equation*}
\zeta_X(t) = \prod_i \bigl( 1 - t^{|\gamma_i|} \bigr)^{-1},
\end{equation*}
where $\{\gamma_i\}_i$ is an enumeration of all the finite orbits in $\cA^{\Z^d}$ (see Section \ref{Sect:Preliminaries} for precise definitions). When $d = 1$, the function $\zeta_X(t)$ is the Artin-Mazur zeta function (see \cite{LindMarcus}); however, for $d > 1$, the function $\zeta_X(t)$ defined here may differ from the zeta function for $\Z^d$-actions defined in \cite{LindZeta}.
\begin{thm} \label{Thm:EmptinessIntro}
Let $\mathcal{A}$ be a finite set, and let $d$ be in $\N$. For each $n$, let $\mathcal{E}_n \subset \Omega_n$ be the event that $X_{\omega} = \varnothing$. Then for each $\alpha$ in $[0,1]$,
\begin{equation*}
\lim_n \PP(\mathcal{E}_n) = \left\{ \begin{array}{ll} 
                                                      \zeta_X(\alpha)^{-1}, & \text{ if } \alpha \in [0,\,|\cA|^{-1}), \\
 						0, & \text{ if } \alpha \in [|\cA|^{-1},\,1].
                                                    \end{array}
			 \right.
\end{equation*}
Furthermore, for $\alpha \neq |\cA|^{-1}$, the rate of convergence to this limit is at least exponential in $n$.
\end{thm}
Notice that for $d >1$, despite the fact that there is no algorithm to decide in finite time whether a given $\Z^d$-SFT $X_{\omega}$ is empty, we may nonetheless compute the limit of the probability of the event $\mathcal{E}_n$. In particular, the limit of the probability of emptiness is positive for $\alpha < |\cA|^{-1}$ and equal to zero for $\alpha \geq |\cA|^{-1}$. A general discussion of random constraint satisfaction problems is beyond the scope of this work; however, let us note that in the language of random constraint satisfaction problems, Theorem \ref{Thm:EmptinessIntro} identifies $\alpha = |\cA|^{-1}$ as the satisfiability threshold for this model of random $\Z^d$-SFTs. 

Our second main result concerns the limiting distribution of entropy as $n$ tends to infinity. Let $h(X)$ denote the (topological) entropy of a $\Z^d$-SFT $X$:
\begin{equation*}
h(X) = \lim_k \frac{1}{k^d} \log \# \bigl\{ x|_{F_k} : x \in X \bigr\}.
\end{equation*}
Also, let $\log^+(x) = \max(0,\log(x))$.
\begin{thm} \label{Thm:EntropyIntro}
 Let $\mathcal{A}$ be a finite set and $d$ be in $\N$. For each $\alpha \in [0,1]$ and $\epsilon >0$, there exist $\rho>0$ and $n_0$ such that if $n \geq n_0$, then
\begin{equation*}
 \PP \biggl( \Bigl| h(X_{\omega}) - \log^+(\alpha |\cA|) \Bigr| \geq \epsilon \biggr) \leq \exp\Bigl(-\rho n^d\Bigr).
\end{equation*}
\end{thm}
In other words, the distribution of entropy for random $\Z^d$-SFTs converges in probability to a point mass at $\log^+(\alpha |\cA|)$, and this convergence is at least exponential in $n$. From the point of view of constraint satisfaction problems, Theorem \ref{Thm:EntropyIntro} may be interpreted as describing ``how many'' solutions exist, as entropy provides a natural notion of the ``size'' of the solution space. 

Observe that Theorems \ref{Thm:EmptinessIntro} and \ref{Thm:EntropyIntro} hold for any $d \geq 1$. Thus, from the perspective of emptiness 
and entropy, it appears that random $\Z^d$-SFTs behave similarly to random $\Z$-SFTs, despite the fact that the class  of $\Z^d$-SFTs for $d > 1$ exhibits strikingly different behavior than the class of $\Z$-SFTs. 

The following two theorems further address the possible differences between $\Z$-SFTs and $\Z^d$-SFTs. Recall that any nonempty $\Z$-SFT contains a finite orbit, but for each $d>1$, there exist nonempty $\Z^d$-SFTs that contain no finite orbits. The following theorem shows that the probability that a random $\Z^d$-SFT is nonempty but contains no finite orbit tends to zero as $n$ tends to infinity.
\begin{thm} \label{Thm:ExoticBehavior}
Let $\mathcal{A}$ be a finite set and $d$ be in $\N$. For each $n$, let $\mathcal{G}_n \subset \Omega_n$ be the event that $X_{\omega}$ is nonempty but contains no finite orbits. Then for each $\alpha$ in $[0,1]$,
\begin{equation*}
\lim_n \PP( \mathcal{G}_n ) = 0.
\end{equation*}
Furthermore, for $\alpha \neq |\cA|^{-1}$, the rate of convergence to this limit is at least exponential in $n$.
\end{thm}
Note that the undecidability result of Berger \cite{Berger1966}, along with many other constructions of exotic or pathological examples, relies on constructing $\Z^2$-SFTs in $\mathcal{G}_n$ for some $n$. Here we observe the heuristic phenomenon, observed in other random constraint satisfaction problems, that for $d >1$, exotic or pathological behavior is possible ``in the worst case,'' but such behavior is not typical. 

We now consider the periodic entropy of $\Z^d$-SFTs. For a $\Z^d$-SFT $X$, let $\mathcal{P}$ be the set of points $x$ in $X$ such that $x$ is contained in a finite orbit. Then the periodic entropy of $X$ is defined as
\begin{equation*}
h_{\text{per}}(X) = \limsup_{k \to \infty} \frac{1}{k^d} \log | \{ x|_{F_k} : x \in \mathcal{P} \} |.
\end{equation*}
The following theorem gives the limiting distribution of periodic entropy for random $\Z^d$-SFTs.
\begin{thm} \label{Thm:PeriodicEntropyIntro}
 Let $\mathcal{A}$ be a finite set and $d$ be in $\N$. For each $\alpha \in [0,1]$ and $\epsilon >0$, there exist $\rho>0$ and $n_0$ such that if $n \geq n_0$, then
\begin{equation*}
 \PP \biggl( \Bigl| h_{\text{per}}(X_{\omega}) - \log^+(\alpha |\cA|) \Bigr| \geq \epsilon \biggr) \leq \exp\Bigl(-\rho n^d\Bigr).
\end{equation*}
\end{thm}
For any $\Z$-SFT $X$, we have that $h_{\text{per}}(X) = h(X)$, but for $\Z^d$-SFTs with $d>1$, this equality need not hold. Nonetheless, Theorems \ref{Thm:EntropyIntro} and \ref{Thm:PeriodicEntropyIntro} show that for large $n$, in a typical $\Z^d$-SFT defined by forbidding blocks of size $n$, the growth rate of the number finite orbits is about the same as the growth rate of the number of all orbits, which once again suggests that typical $\Z^d$-SFTs are not dramatically different from typical $\Z$-SFTs.

The paper is organized as follows. In Section \ref{Sect:Preliminaries} we present the necessary notation and define random $\Z^d$-SFTs in detail. Section \ref{Sect:Emptiness} deals with emptiness and finite orbits and contains the proofs of Theorems \ref{Thm:EmptinessIntro} and \ref{Thm:ExoticBehavior}. That section is the longest and most difficult of the paper, reflecting the difficulty of studying emptiness for $\Z^d$-SFTs with $d>1$. Indeed, our proofs in Section \ref{Sect:Emptiness} employ fundamentally different methods from those used in \cite{McGoff} for the case $d=1$. Section \ref{Sect:Entropy} details our considerations related to entropy, as well as the proofs of Theorems \ref{Thm:EntropyIntro} and \ref{Thm:PeriodicEntropyIntro}. These proofs use a second moment argument similar in structure to the one given in \cite{McGoff} for $d=1$, although the proofs here differ from the previous proofs in the details. Finally, we conclude the paper with some brief remarks and open questions in Section \ref{Sect:Discussion}.


\section{Preliminaries} \label{Sect:Preliminaries}

In this section, we give a precise description of the objects under consideration. 

\subsection{Basic notation} \label{Sect:Notation}

\subsubsection{General subsets of $\Z^d$}
Let $d$ be a natural number. First, we set some notation regarding subsets of $\Z^d$. Let $\{e_i\}_{i = 1}^d$ denote the standard basis in $\Z^d$. We use interval notation to denote subsets of $\Z$, \textit{e.g.}, $[1,4] = \{1,2,3,4\}$.  For a subset $E$ of $\Z^d$ and a vector $v$ in $\Z^d$, let $E+v = \{u+v : u \in E\}$. Also, let $\1 = (1,\dots,1) \in \Z^d$. We use the $\ell_{\infty}$ metric: for $u,v$ in $\R^d$, let $\rho(u,v) = \max \{ |u_i - v_i|: 1 \leq i \leq d\}$. 
For $x$ in $\R^d$ and $r \geq 0$, we let $B(x,r)$ denote the closed ball centered at $x$ of radius $r$ in $\Z^d$. 
We will have use for the inner boundary of a set: for $E \subset \Z^d$ and $r >0$, define
\begin{align*}
  \partial_r E & = \{ x \in E : \exists y \notin E, \, \rho(x,y) \leq r\}.
 \end{align*}
We denote the $r$-interior of a set $E$ by
\begin{equation*}
 \interior_r(E) = \{ x \in E : B(x,r) \subset E \},  
\end{equation*}
and we denote the $r$-thickening of $E$ by
\begin{equation*}
 B(E,r) = \bigcup_{x \in E} B(x,r). 
\end{equation*}
Let ``$<$'' denote the lexicographic ordering on $\Z^d$: for $p \neq q$ in $\Z^d$, let $i = \min\{ j : p_j \neq q_j \}$, and define $p < q$ whenever $p_i < q_i$. Also, for each $i$ in $[1,d]$ and $v \subset \Z^d$, let $\pi_i(v)$ denote the projection of $v$ to the hyperplane passing through $\1$ perpendicular to $e_i$, \textit{i.e.}, $(\pi_i(v))_i = 1$ and $(\pi_i(v))_j = v_j$ for $j\neq i$.

\subsubsection{Rectangles and hypercubes} \label{Sect:Rectangles}
We also require some definitions regarding the geometry of rectangles and hypercubes in $\Z^d$. For $n$ in $\N$, let $F_n$ be the hypercube with side length $n$ in $\Z^d$, \textit{i.e.}, $F_n = [1,n]^d \subset \Z^d$. Let $k$ be in $\N$. For each $\mathcal{I} \subset [1,d]$ and $s : \mathcal{I} \to \{1,k\}$, define
\begin{equation*}
 F_k(\mathcal{I},s) = \{ (x_1,\dots,x_d) \in F_k : \forall i \in \mathcal{I}, \, x_i = s(i) \}.
\end{equation*}
A set of the form $F_k(\mathcal{I},s)$ is called a \textit{face} of $F_k$, and we define the \textit{dimension} of $F_k(\mathcal{I},s)$ to be $\dim(F_k(\mathcal{I},s)) = d-|\mathcal{I}|$.  For example, if $\mathcal{I} = \varnothing$ (and $s : \varnothing \to \{1,k\}$ is the empty map), then $F_k(\varnothing,s) = F_k$, which has dimension $d$. 
For $0 \leq \ell \leq d$, define $F_{k,\ell}$ to be the \textit{$\ell$-skeleton} of $F_k$:
\begin{equation*}
 F_{k,\ell} = \bigcup_{\substack{E \text{ face of } F_k \\ \dim(E) = \ell}} E.
\end{equation*}
Note that the number of faces of dimension $\ell$ is $2^{d-\ell} \binom{d}{\ell}$, which we denote by $c_{d,\ell}$. 

For a face $E = F_k(\mathcal{I},s)$, we also require the ``$n$-thickened interior'' of $E$:
\begin{equation*} 
\begin{split}
  T_n(E) = \biggl\{ p \in F_k : \,  & \forall i \in \mathcal{I}, \, |s(i) - p_i| \leq n, \text{ and } \\
                         & \forall i \notin \mathcal{I}, \, p_i \in [n+1,k-n] \biggr\}.
\end{split}
\end{equation*}
A face $E$ has some restricted coordinates (indexed by $\mathcal{I}$) and some free coordinates (those not in $\mathcal{I}$). In this sense, $T_n(E)$ consists of the $n$-thickening of $E$ in the restricted directions and the $n$-interior of $E$ in the free directions (see Figure \ref{TnEPartition}). For $k > 2n$, the collection of sets $T_n(E)$, where $E$ ranges over all faces of $F_k$, is a partition of $F_k$ (see Figure \ref{TnEPartition}). Furthermore, note that for any $\ell \in [0,d]$, we have
\begin{equation} \label{Eqn:ellSkeletonDecomp}
 F_k = \biggl( F_k \cap B(F_{k,\ell},n) \biggr) \cup \Biggl( \bigsqcup_{\substack{E \text{ face of } F_k \\ \dim(E) > \ell }} T_n(E) \Biggr).
\end{equation}

\begin{figure}
 \centering
\includegraphics[scale=.5]{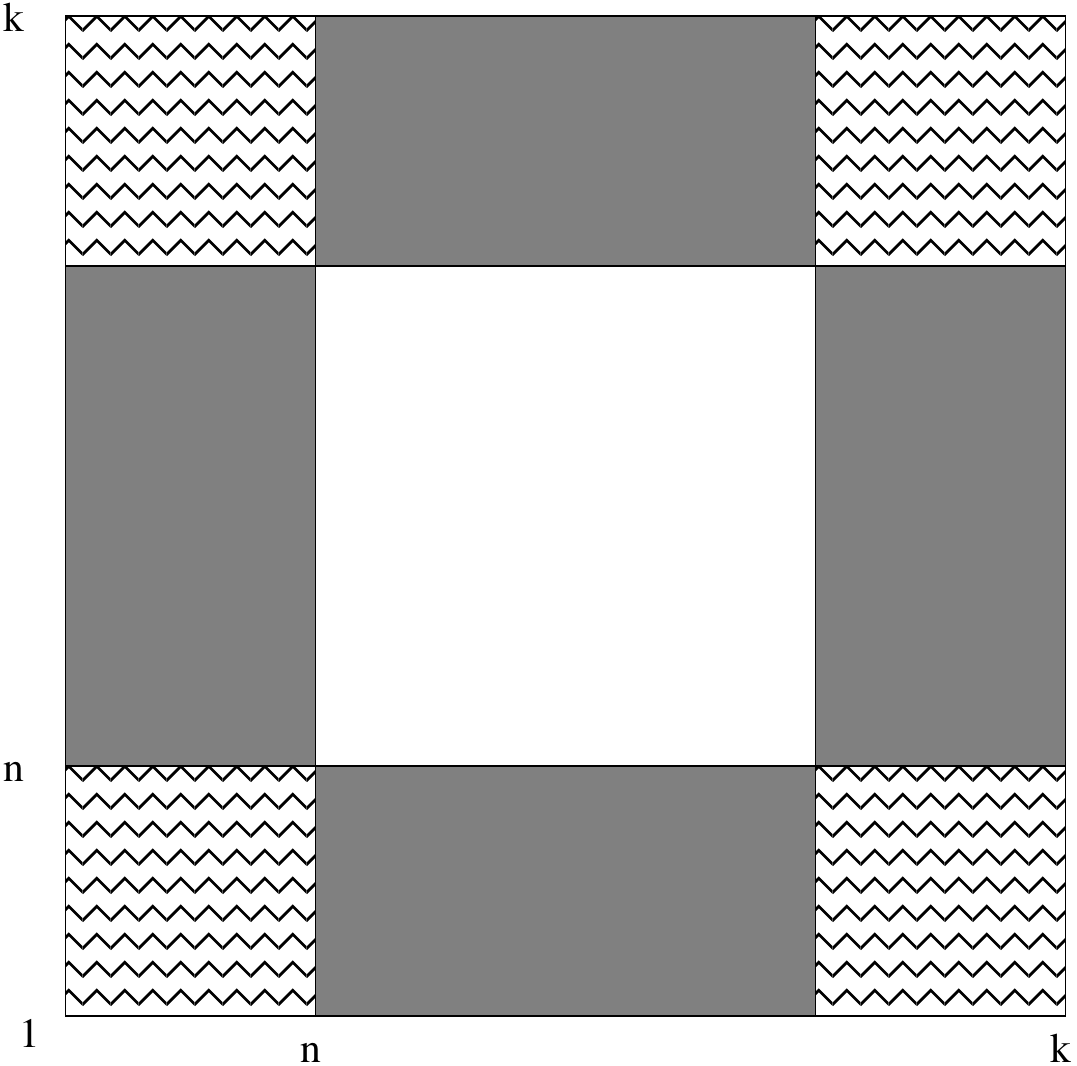}
\caption{Partition of $F_k$ by the sets $T_n(E)$, where $E$ runs over all faces. For each corner $E$ (\textit{i.e.}, face of dimension $0$), $T_n(E)$ is shaded with broken lines. For each edge $E$ (\textit{i.e.} face of dimension $1$), $T_n(E)$ is shaded in gray. For the square $F_k$ (\textit{i.e.}, the only face of dimension $2$), $T_n(F_k)$ is the unshaded central region.}
\label{TnEPartition}
\end{figure}


\subsubsection{Patterns, repeats, and repeat covers} \label{Sect:Repeats}

Let us establish some basic terminology. 
A \textit{configuration} is an element of $\cA^{\Z^d}$. For a finite set $E \subset \Z^d$ and a finite set $\cA$, a \textit{pattern} on $E$ is an element of $\cA^E$. If $u$ is a pattern on $E$, then we say that $u$ has shape $E$. A pattern $u$ on $E$ may also be referred to as an $E$-pattern or a pattern with shape $E$. For ease of notation, we consider patterns to be defined only up to translation, as follows. If $F = E + v$ for some $v$ in $\Z^d$ and $u$ is in $\cA^{S}$ where $E \cup F \subset S$, then we write $u|_F = u|_E$ to denote the statement that $u_{t+v} = u_t$ for all $t$ in $E$.

For finite $E \subset \Z^d$, let $\cC_n(E)$ denote the set of $n$-cubes contained in $E$, \textit{i.e.},
\begin{equation*}
 \cC_{n}(E) = \{ F_n + v : v \in \Z^d, \, F_n + v \subset E \}.
\end{equation*}
For a finite set $S \subset \Z^d$, let $m(S)$ be the lexicographically minimal element of $S$. Note that for an $n$-cube $S$, we have the relation $S = F_n+m(S)-\1$. 

In the following, we will work with sets of pairs of $n$-cubes; that is, we will work with sets $J \subset \cC_n(E) \times \cC_n(E)$ for various choices of $E \subset \Z^d$. For such a set $J$, let $|J|$ denote the number of pairs in $J$, and let $\rArea{J} \subset E$ be defined as follows:
\begin{equation*}
 \rArea{J} = \bigcup_{(S_1,S_2) \in J} S_2.
\end{equation*}

We will seek to understand the structure of ``repeated sub-patterns'' within patterns on finite subsets of $\Z^d$. We make this notion precise with the following definition, which plays an important role in Section \ref{Sect:Emptiness}. 

\begin{defn}
Let $\cA$ be a finite set. Let $E \subset \Z^d$, and let $u$ be in $\cA^{E}$. A pair $(S_1,S_2)$ in $\cC_n(E) \times \cC_n(E)$ is an $n$-\textit{repeat} in $u$ if $S_1 \neq S_2$, $u|_{S_1} = u|_{S_2}$, and $S_1$ is the lexicographically minimal appearance of the pattern $u|_{S_1}$ in $u$.
If $(S_1,S_2)$ is an $n$-repeat in $u$, then the pattern $w = u|_{S_1} = u|_{S_2}$ is called a \textit{repeated $F_n$-pattern}. A set $J \subset \cC_n(E) \times \cC_n(E)$ is an $n$-\textit{repeat cover} for $u$ if
\begin{enumerate}
\item each $(S_1,S_2) \in J$ is an $n$-repeat in $u$, and
\item if $(S_1,S_2)$ is an $n$-repeat in $u$, then $S_2 \subset \rArea{J}$.
\end{enumerate}
When $n$ is clear from context, we drop the prefix $n$ from this terminology.
\end{defn}
Let us make two remarks about this definition. First, if $J$ is a repeat cover for $u$, then the set $\rArea{J}$ depends only on $u$ and not on the particular choice of repeat cover $J$. Second, repeat covers always exist, since the set of all pairs in $\cC_n(E) \times \cC_n(E)$ satisfying property (1) gives an $n$-repeat cover. However, we will often seek ``more efficient'' repeat covers, in the sense of including fewer repeats (\textit{i.e.}, minimizing $|J|$). This pursuit is taken up in Section \ref{Sect:RepeatCovers}.

\subsection{Symbolic dynamical systems}
Now we introduce symbolic dynamical systems and $\Z^d$-SFTs. Let $\mathcal{A}$ be a finite set (alphabet).  For any set $Y$ contained in $\mathcal{A}^{\Z^d}$, let $W_n(Y)$ be the set of $F_n$-patterns seen in $Y$:
\begin{align*}
W_n(Y) & = \{ u \in \mathcal{A}^{F_n} : \exists x \in Y, \, \exists v \in \Z^d, \; u = x|_{F_n+v} \}.
\end{align*}
Endow $\mathcal{A}^{\Z^d}$ with the product topology induced by the discrete topology on $\mathcal{A}$, and let $\sigma : \mathcal{A}^{\Z^d} \to \mathcal{A}^{\Z^d}$ denote the shift action, defined for each $x$ in $\mathcal{A}^{\Z^d}$ and $p,q$ in $\Z^d$ by
\begin{equation*}
\bigl(\sigma^p(x)\bigr)_q = x_{p+q}.
\end{equation*}
Note that $\mathcal{A}^{\Z^d}$ is a compact, metrizable space, and $\sigma^p$ is a homeomorphism for each $p$ in $\Z^d$. A set $Y$ contained in $\mathcal{A}^{\Z^d}$ is a $\Z^d$-\textit{subshift} if $Y$ is closed and shift-invariant (\textit{i.e.}, $\sigma^p(Y) = Y$ for all $p$ in $\Z^d$). A set $X$ contained in $\mathcal{A}^{\Z^d}$ is a $\Z^d$-\textit{shift of finite type} ($\Z^d$-SFT) if there exist a natural number $n$ and a set $\mathcal{F} \subset \mathcal{A}^{F_n}$ such that
\begin{equation*}
X = \bigl\{ x \in \mathcal{A}^{\Z^d} : \forall p \in \Z^d, \, \sigma^p(x)|_{F_n} \notin \mathcal{F} \bigr\}.
\end{equation*}
Any $\Z^d$-SFT is a $\Z^d$-subshift, but the converse is false.

For $S \subset \Z^d$ and a pattern $u \in \mathcal{A}^{S}$, define $W_n(u)$ to be the set of $F_n$-patterns that appear in $u$:
\begin{equation*}
W_n(u) = \{ v \in \mathcal{A}^{F_n} : \exists S' \in \cC_n(S), \, u|_{S'} = v\}.
\end{equation*}
For a fixed alphabet $\mathcal{A}$, if $n \leq k$ and $j \in [1, \, (k-n+1)^d]$, then we define $N_{n,k}^j$ to be the number of patterns with shape $F_k$ containing exactly $j$ distinct $F_n$-patterns:
\begin{equation} \label{Eqn:Nnkj}
N_{n,k}^j = \{u \in \cA^{F_k} : |W_n(u)|=j\}.
\end{equation}

Let $X$ be a $\Z^d$-subshift. A \textit{finite orbit} in $X$ is a nonempty finite set $\gamma \subset X$ such that for each $x$ in $\gamma$, it holds that $\{\sigma^p(x) : p \in \Z^d\} = \gamma$. A point $x$ in $X$ is a \textit{totally periodic point} if $x$ is contained in a finite orbit. Let $\{\gamma_i\}_i$ be an enumeration of the countable set of finite orbits contained in $X$ such that if $i < j$ then $|\gamma_i|\leq |\gamma_j|$. We formally define a zeta function for $X$ as follows:
\begin{equation*}
\zeta_X(t) = \prod_i \bigl(1-t^{|\gamma_i|} \bigr)^{-1}.
\end{equation*}
Simple computations show that if $X = \cA^{\Z^d}$, then the radius of convergence of $\zeta_X(t)$ is $|\cA|^{-1}$ and $\zeta_X(t)$ diverges to infinity at $t = |\cA|^{-1}$.

\subsection{Random $\Z^d$-SFTs} \label{Sect:StochasticModel}

Consider a fixed alphabet $\cA$ and a natural number $d$. Let $X = \cA^{\Z^d}$. For each $\alpha \in [0,1]$ and $n$ in $\N$, we make the following definitions. Set $\Omega_n = \{0,1\}^{W_n(X)}$. For each $u$ in $W_n(X)$, let $\xi_u : \Omega_n \to \{0,1\}$ be the projection onto the $u$-coordinate. Let $\PP$ be the product measure on $\Omega_n$ defined for each $u$ in $W_n(X)$ by
\begin{equation*}
\PP( \xi_u = 1 ) = \alpha.
\end{equation*}

We view $\PP$ as a probability measure on the $\Z^d$-SFTs contained in $\mathcal{A}^{\Z^d}$, which we explain as follows. For $n \geq 1$ and $\omega$ in $\Omega_n$, let
\begin{align*}
\mathcal{F}(\omega) & = \{ u \in W_n(X) : \xi_u(\omega) = 0 \}.
\end{align*}
Then for each $\omega$ in $\Omega_n$, let $X_{\omega}$ be the $\Z^d$-SFT built from $\omega$ by forbidding the patterns in $\mathcal{F}(\omega)$:
\begin{align*}
X_{\omega} & = \{ x \in X : \forall p \in \Z^d, \, \sigma^p(x)|_{F_n} \notin \mathcal{F}(\omega) \} .
\end{align*}
This definition is equivalent to the description given in the introduction.

Let $K_n(X)$ denote the finite set of $\Z^d$-SFTs contained in $X$ that may be defined by forbidding only patterns in $W_n(X)$. We include the empty $\Z^d$-SFT as an element of $K_n(X)$ for all $n$. Then the map $\omega \mapsto X_{\omega}$ is a surjection from $\Omega_n$ onto $K_n(X)$. Therefore $\PP$ projects to a probability measure on $K_n(X)$, and it is in this sense that we refer to $X_{\omega}$ as a \textit{random} $\Z^d$-SFT. 

\subsection{Two simple combinatorial lemmas}

We will use the following two elementary lemmas, whose proofs are included for completeness.

\begin{lemma} \label{Lemma:CombLemma}
 Suppose $S$ is a finite set and $\{U_i: i \in \mathcal{I}\}$ is a collection of subsets of $S$ such that for each $s$ in $S$,
\begin{equation*}
 |\{ i \in \mathcal{I} : s \in U_i \} | \leq c.
\end{equation*}
Then
\begin{equation*}
 \sum_{i \in \mathcal{I}} |U_i| \leq c |S|.
\end{equation*}
\end{lemma}
\begin{proof}
 Let $\tilde{S} = \{(i,s) \in \mathcal{I} \times S : s \in U_i \}$. Then we have
\begin{align*}
 \sum_{i \in \mathcal{I}} |U_i| & = |\tilde{S}|  = \sum_{s \in S} |\{(i,s) : s \in U_i \}|  \leq c |S|.
\end{align*}
\end{proof}


\begin{lemma} \label{Lemma:IntervalCoveringLemma}
Let $\{L_i\}_{i \in \mathcal{I}}$ be a finite set of intervals in $\Z$. Then there exists $\mathcal{I}' \subset \mathcal{I}$ such that 
\begin{equation*}
\bigcup_{i \in \mathcal{I}} L_i = \bigcup_{i \in \mathcal{I}'} L_i, 
\end{equation*}
and for each $x$ in $\Z$, it holds that
\begin{equation*}
|\{ i \in \mathcal{I}' : x \in L_i \}| \leq 2.
\end{equation*}
\end{lemma}
\begin{proof}
Let $\mathcal{I}_0 = \mathcal{I}$. Now assume for induction (on $m$) that $\mathcal{I}_m$ is defined. If there exist $i_1, i_2, i_3 \in \mathcal{I}_m$ such that $L_{i_1} \subset L_{i_2} \cup L_{i_3}$, then let $\mathcal{I}_{m+1} = \mathcal{I}_m \setminus \{i_1\}$. As $\mathcal{I}$ is finite, this process eventually halts, resulting in a set $\mathcal{I}' \subset \mathcal{I}$. Note that the union of intervals in each $\mathcal{I}_m$ is the same as the union of the intervals in $\mathcal{I}$. Also, if there exists $x$ in $\Z$ contained in three distinct intervals in $\mathcal{I}'$, then at least one of those intervals is contained in the union of the other two, which contradicts the definition of $\mathcal{I}'$. Hence $\mathcal{I}'$ has the desired properties.
\end{proof}

\section{Emptiness and finite orbits} \label{Sect:Emptiness}

In this section, we investigate the probability of the event that the random $\Z^d$-SFT is empty, an event that we refer to simply as \textit{emptiness}. As evidenced by the presence of the zeta function in Theorem \ref{Thm:EmptinessIntro}, the probability of emptiness is intimately related to the probabilities associated to allowing each of the finite orbits. For an idea about how this connection might arise, consider the basic observation that any finite orbit $\gamma$ is contained in the random $\Z^d$-SFT with probability $\alpha^{|W_n(\gamma)|}$, since $\gamma$ is in $X_{\omega}$ precisely when each of the $F_n$-patterns in $\gamma$ is allowed. If $n$ is large relative to $|\gamma|$, then $|\gamma| = |W_n(\gamma)|$. In this case, the probability of forbidding $\gamma$ is $1- \alpha^{|\gamma|}$. Now if each of the finite orbits were forbidden independently (which is most certainly \textit{not} the case), then the probability of forbidding all the finite orbits would be the infinite product $ \displaystyle \prod_i (1-\alpha^{|\gamma_i|}) =  \zeta_X(\alpha)^{-1}$. Furthermore, if emptiness were equivalent to forbidding all the finite orbits (which is also not the case), then we would obtain that the probability of emptiness equals $\zeta_X(\alpha)^{-1}$. The proof of Theorem \ref{Thm:EmptinessIntro} involves showing that despite the fact that above heuristics are not true, the result still holds asymptotically as $n$ tends to infinity.

Let us give a brief outline of this section. To find the limiting behavior of the probability of emptiness, we find upper and lower bounds on this probability for finite $n$ and then deduce the limiting behavior from these bounds. The upper bound, which is not difficult, appears in Section \ref{Sect:EmptinessUB}. Note that the upper bound alone implies Theorem \ref{Thm:EmptinessIntro} in the case $\alpha \geq |\cA|^{-1}$, since the upper bound tends to zero in that regime. The most difficult part of the paper involves finding an appropriate lower bound for the probability of emptiness in the sub-critical regime, $\alpha < |\cA|^{-1}$. Towards that end, we investigate the behavior of finite orbits in Section \ref{Sect:PeriodicBehavior}. In Sections \ref{Sect:BasicsOfRepeats} and \ref{Sect:RepeatCovers}, we prove several lemmas that help us show that the dominant contribution to the probability of emptiness comes from the finite orbits; that is, if the finite orbits are forbidden from the random $\Z^d$-SFT, then with high probability, all other orbits are forbidden as well. Finally, we put these pieces together in Section \ref{Sect:EmptinessProof} and prove Theorems \ref{Thm:EmptinessIntro} and \ref{Thm:ExoticBehavior}.

\subsection{Upper bound on probability of emptiness} \label{Sect:EmptinessUB}

The purpose of this section is to present Proposition \ref{Prop:UB}, which gives upper bounds on the probability of emptiness. For $X = \cA^{\Z^d}$, by \cite[Remark 6.3]{McGoff}, the following upper bound holds, with $R_{per}$ as the radius of convergence of $\zeta_X$ and $\mathcal{E}_n$ as the event $X_{\omega} = \varnothing$:
\begin{equation*} 
\limsup_n \PP(\mathcal{E}_n) \leq \left\{ \begin{array}{ll} 
                                                      \zeta_X(\alpha)^{-1}, & \text{ if } \alpha \in [0,R_{per}) \\
																											0, & \text{ if } \alpha \in [R_{per},1].
                                                    \end{array}
																						 \right.
\end{equation*}
This bound is sufficient to give the correct upper bound on the limit in Theorem \ref{Thm:EmptinessIntro}, but in order to establish a rate of convergence of $\PP(\mathcal{E}_n)$ to its limit, we use the more detailed estimates in Proposition \ref{Prop:UB}.

Before we prove Proposition \ref{Prop:UB}, we require the following lemma, which gives a condition under which two distinct finite orbits have disjoint sets of $F_n$-patterns. The importance of this lemma is that if a collection of orbits has disjoint sets of $F_n$-patterns, then these orbits behave independently with respect to $\PP$.
\begin{lemma} \label{Lemma:Independence}
 Suppose $\gamma_1 \neq \gamma_2$ are finite orbits such that $|\gamma_i| \leq n/2$ for $i = 1,2$. Then  $W_n(\gamma_1) \cap W_n(\gamma_2) = \varnothing$. 
\end{lemma}
\begin{proof}
 Suppose for contradiction that $\gamma_1$ and $\gamma_2$ are finite orbits with $|\gamma_i| \leq n/2$ and $W_n(\gamma_1) \cap W_n(\gamma_2) \neq \varnothing$. We would like to show that $\gamma_1 = \gamma_2$. Let $u$ be in $W_n(\gamma_1) \cap W_n(\gamma_2)$, and let $x^i$ be in $\gamma_i$ such that $x^i|_{F_n-\1} = u$. Let us show that $x^1 = x^2$, which implies that $\gamma_1 = \gamma_2$ (since they are assumed to be finite orbits).

 Let $p = (p_1,\dots,p_d)$ be in $\Z^d$. Let us show that $x^1(p) = x^2(p)$. Define $q_0 = 0 \in \Z^d$, and $q_k = \sum_{j=1}^k p_j e_j$, for $k = 1, \dots, d$. Let us prove by induction on $k$ that $\sigma^{q_k}(x^1)|_{F_n-\1} = \sigma^{q_k}(x^2)|_{F_n-\1}$ for $k=0,\dots,d$. Since $x^1|_{F_n-\1} = u = x^2|_{F_n-\1}$ by construction, the statement holds for $k = 0$. Now suppose that $\sigma^{q_k}(x^1)|_{F_n-\1} = \sigma^{q_k}(x^2)|_{F_n-\1}$ for some $k < d$. Recall that $ \pi_{k+1}(F_n) = \{v \in F_n : v_{k+1} = 1\}$. For $i = 1, 2$, define $f_i : \Z \to \cA^{\pi_{k+1}(F_n)}$ by
\begin{equation*}
 f_i(m) = \sigma^{q_k}(x_i)|_{\pi_{k+1}(F_n) - \1 + m e_{k+1}}.
\end{equation*}
Since $\gamma_i$ is a finite orbit and $|\gamma_i| \leq n/2$, we have that $f_i$ is periodic with period less than or equal to $n/2$. Furthermore, by the inductive hypothesis, $f_1(m) = f_2(m)$ for $m = 0, \dots, n-1$. Thus, by the Fine-Wilf Theorem \cite{FineWilf}, we have that $f_1 = f_2$. In particular, we have shown that $\sigma^{q_{k+1}}(x^1)|_{F_n-\1} = \sigma^{q_{k+1}}(x^2)|_{F_n-\1}$, which completes the inductive step. Hence, $\sigma^{p}(x^1)|_{F_n-\1} = \sigma^{p}(x^2)|_{F_n-\1}$, which gives in particular that $x^1(p) = x^2(p)$. Since $p$ was arbitrary, we conclude that $x^1 = x^2$, which finishes the proof.
\end{proof}

The following proposition gives upper bounds on the asymptotic behavior of the probability of emptiness. The main idea of the proof is to use Lemma \ref{Lemma:Independence} to find a large set of finite orbits which are independent with respect to $\PP$.
\begin{prop} \label{Prop:UB}
Let $\mathcal{A}$ be a finite set, $d$ be in $\N$, and $X = \mathcal{A}^{\Z^d}$.
Let $\Per(X_{\omega})$ be the set of finite orbits in $X_{\omega}$. Then for each $\alpha$ in $[0,1]$ and any $n$,
\begin{equation} \label{Eqn:PerUB}
  \PP \bigl( \Per(X_{\omega}) = \varnothing \bigr) \leq \prod_{|\gamma| \leq n/2} \bigl(1-\alpha^{|\gamma|} \bigr),
\end{equation}
where the product runs over the finite orbits $\gamma$ in $X$ satisfying $|\gamma| \leq n/2$.
Furthermore, for each $\alpha < |\cA|^{-1}$, there exist $C>0$ and $\beta \in (0,1)$ such that for  large enough $n$,
\begin{equation} \label{Eqn:ExpUBsubcrit}
 \PP( \Per(X_{\omega}) = \varnothing ) \leq \zeta_X(\alpha)^{-1}( 1 + C \beta^n),
\end{equation}
and for each $\alpha > |\cA|^{-1}$, there exist $C >0$ and $\beta >1$ such that for large enough $n$,
\begin{equation} \label{Eqn:ExpUBsupercrit}
 \PP ( \Per(X_{\omega}) = \varnothing ) \leq \exp( - C \beta^n).
\end{equation}
\end{prop}
\begin{proof}
Let $X$ and $\Per(X_{\omega})$ be as above, and let $\alpha$ be in $[0,1]$. For a finite orbit $\gamma$ in $X$, let $F(\gamma)$ be the event that $\gamma$ is forbidden in $X_{\omega}$ (\textit{i.e.}, $\gamma \notin \Per(X_{\omega})$). By Lemma \ref{Lemma:Independence}, if $\gamma_1 \neq \gamma_2$ and $|\gamma_i| \leq n/2$ for $i = 1,2$, then $W_n(\gamma_1) \cap W_n(\gamma_2) = \varnothing$, and therefore the events $F(\gamma_1)$ and $F(\gamma_2)$ are independent. In fact, the entire collection of events $\{F(\gamma) : |\gamma| \leq n/2\}$ is independent. Thus, by inclusion, independence, and the fact that $\PP( F(\gamma)) = 1- \alpha^{|W_n(\gamma)|} = 1- \alpha^{|\gamma|}$ whenever $|\gamma|\leq n/2$, we have
\begin{align*}
\begin{split} 
 \PP( \Per(X_{\omega}) = \varnothing ) & \leq \PP \Biggl( \bigcap_{|\gamma|\leq n/2} F(\gamma) \Biggr) \\
 & = \prod_{|\gamma|\leq n/2} \PP\Bigl( F(\gamma) \Bigr) \\
 &  = \prod_{|\gamma| \leq n/2} \Bigl(1-\alpha^{|\gamma|} \Bigr),
\end{split}
\end{align*}
which proves (\ref{Eqn:PerUB}).

Now suppose $\alpha <  |\cA|^{-1}$ (so that $\alpha |\cA| < 1$). Choose $\lambda$ such that $|\cA| < \lambda$ and $\alpha \lambda < 1$. 
Let $P_j$ be the set of finite orbits $\gamma$ in $\mathcal{A}^{\Z^d}$ such that $|\gamma| = j$. Then by (\ref{Eqn:PerUB}) and the fact that $0< \zeta_X(\alpha) < \infty$ for $\alpha < |\cA|^{-1}$, we have
\begin{align} 
 \begin{split} \label{Eqn:GoHeels}
  \PP( \Per(X_{\omega}) = \varnothing ) & \leq \prod_{|\gamma| \leq n/2} \Bigl(1-\alpha^{|\gamma|}\Bigr) \\
 & =  \zeta_X(\alpha)^{-1}  \exp \Biggl( - \sum_{|\gamma| > n/2} \log ( 1- \alpha^{|\gamma|})  \Biggr) \\
 & =  \zeta_X(\alpha)^{-1} \exp \Biggl( - \sum_{j > n/2}  |P_j| \log ( 1- \alpha^j)  \Biggr).
\end{split}
\end{align}
Note that for all $j$ in $\N$, we have $|P_j| \leq j^{d+1} |\cA|^j$ (see \cite[Proposition 4.3]{LindZeta}). Thus, there exists $n_0$ such that for $j > n_0/2$, it holds that $|P_j| \leq \lambda^j$. Let $\beta_1 = (\alpha \lambda)^{1/2}$, and note that $\beta_1 < 1$ (since $\alpha \lambda < 1$).
By calculus, there exist $C_1,C_2>0$ and $n_1 \geq n_0$ such that if $n \geq n_1$ and $j >n/2$, then $-\log(1-\alpha^j) \leq C_1 \alpha^j$ and $\exp( C_1 (1-\alpha \lambda)^{-1} \beta_1^n) \leq 1 + C_2 \beta_1^n$. Then for $n \geq n_1$, using these facts and (\ref{Eqn:GoHeels}), we obtain
\begin{align*}
\begin{split} 
  \PP \bigl( \Per(X_{\omega} \bigr) = \varnothing ) & \leq \zeta_X(\alpha)^{-1} \exp \biggl( - \sum_{j > n/2}  |P_j| \log ( 1- \alpha^j)  \biggr) \\
 & \leq \zeta_X(\alpha)^{-1} \exp \biggl( C_1 \sum_{j >n/2} \lambda^j \alpha^j  \biggr) \\
 & = \zeta_X(\alpha)^{-1} \exp \biggl( C_1 \beta_1^n (1-\alpha \lambda)^{-1}  \biggr) \\
 & \leq \zeta_X(\alpha)^{-1}( 1 + C_2 \beta_1^n).
\end{split}
\end{align*}
Taking $C = C_2$ and $\beta = \beta_1$, we obtain (\ref{Eqn:ExpUBsubcrit}).

Now suppose $\alpha > |\cA|^{-1}$ (so that $\alpha |\cA| > 1$). Choose $\lambda < |\cA|$ such that $\alpha \lambda >1$. By (\ref{Eqn:PerUB}), we have
\begin{align}
 \begin{split} \label{Eqn:PerUBdook}
  \PP \bigl( \Per(X_{\omega} \bigr) = \varnothing ) & \leq  \prod_{|\gamma| \leq n/2} \Bigl(1-\alpha^{|\gamma|}\Bigr) \\
  & = \exp \Biggl( \sum_{j \leq n/2}  |P_j| \log ( 1- \alpha^j)  \Biggr).
 \end{split}
\end{align}
To get a lower bound on $|P_j|$, we count the finite orbits of cardinality $j$ that are constant along directions $e_2, \dots, e_d$. A point in such an orbit can be obtained by concatenating a word of length $w$ of length $j$ along direction $e_1$, provided that the infinite sequence of repeated $w$'s has least period $j$. There are at least $|\cA|^j/2$ words $w$ that with that property, which shows that there are at least $|\cA|^j/(2j)$ finite orbits of size $j$ obtained in this way. Hence, $|P_j| \geq |\cA|^j/  (2j)$, and therefore there exists $n_2$ such that if $j \geq \lfloor n_2/2 \rfloor $, then $\lambda^j < |P_j|$. Also, there exist $C_3>0$ and $n_3 \geq n_2$ such that if $j \geq \lfloor n_3/2 \rfloor$, then $\log(1-\alpha^j) \leq -C_3 \alpha^j$. Hence, for $n \geq n_3$, by these facts and (\ref{Eqn:PerUBdook}), we have
\begin{align*}
 \begin{split}
  \PP \bigl( \Per(X_{\omega} \bigr) = \varnothing ) & \leq \exp \Biggl( \sum_{j \leq n/2}  |P_j| \log ( 1- \alpha^j)  \Biggr) \\
  & \leq \exp \biggl( \lambda^{\lfloor n/2 \rfloor } \log \bigl( 1- \alpha^{ \lfloor n/2 \rfloor}\bigr)  \biggr) \\
  & \leq \exp \biggl( - C_3 (\alpha \lambda)^{\lfloor n/2 \rfloor} \biggr).
 \end{split}
\end{align*}
Setting $C = C_3 (\alpha \lambda)^{-1}$ and $\beta = (\alpha \lambda)^{1/2}$ completes the proof of (\ref{Eqn:ExpUBsupercrit}).
\end{proof}

\subsection{Periodic behavior} \label{Sect:PeriodicBehavior}

The ultimate goal of this section is to prove Lemma \ref{Lemma:Smallj}, which states that if $|W_n(u)|$ is small enough, then there exists a finite orbit $\gamma$ such that $\gamma$ is allowed whenever $u$ is allowed. First, we need some preliminary notation and lemmas. 

For the next few lemmas, we'll consider combinatorics of patterns with $d=1$, which we will call \textit{words}. Suppose $k$ and $n$ are fixed and $k > n$. For a word $u$ in $\cA^k$, let $W'_m(u)$ be the set of words of length $m$ that appear in $u$ at least once with first coordinate not greater than $k-n+1$:
\begin{equation*}
 W'_m(u) = \{ w : \exists 1 \leq i \leq k-n+1, \, u[i,i+m-1] = w\}.
\end{equation*}
The following two lemmas are restatements of results by Morse and Hedlund \cite{MorseHedlund}. Nonetheless, we include proofs for completeness.
\begin{lemma} \label{Lemma:Periodicity1}
Suppose $k > n$, $u \in \cA^k$, and $|W_n(u)| \leq n$. Then at least one of the following conditions holds: (i) there exists $m<n$ such that every word in $W'_m(u)$ has a unique  right extension in $W'_{m+1}(u)$, or (ii) there exists a symbol $b$ in $\cA$ such that $u_t = b$ for $t \in [1,k-n+1]$.
\end{lemma}
\begin{proof}
For $1 \leq m < n$, let $\rho_m : W'_{m+1}(u) \to W'_m(u)$ be defined by removing the last symbol of any word of length $m+1$, \textit{i.e.}, if $w$ is in $\cA^m$ and $a$ is in $\cA$, then $\rho_m(wa) = w$. Then $\rho_m$ is surjective, so $|W'_{m+1}(u)| \geq  |W'_m(u)|$. 

If (ii) holds, then we're done. Suppose (ii) does not hold, and therefore we must have $|W'_1(u)| \geq 2$. If $|W'_{m+1}(u)| = |W'_m(u)|$ for some $m < n$, then $\rho_m$ is a bijection, and hence every word in $W'_m(u)$ has a unique extension in $W'_{m+1}(u)$. Otherwise, $|W'_{m+1}(u)| > |W'_m(u)|$ for $m = 1, \dots, n-1$, and therefore $|W'_n(u)| \geq (n-1) + |W'_1(u)| \geq n+1$, which contradicts $|W'_n(u)| \leq |W_n(u)|  \leq n$. We conclude that if (ii) does not hold, then (i) must hold.
\end{proof}

In the following lemma, we show that words $u$ (still in the setting $d = 1$) satisfying $|W_n(u)| \leq n$, can be decomposed into a prefix, a periodic part, and a suffix, and we give bounds on the lengths of these parts.
\begin{lemma} \label{Lemma:Periodicity2}
 Suppose $k > 3 n$ and $|W_n(u)| = j \leq n$. Then the word $u[n,k-n]$ is periodic with period not greater than $j$. 
\end{lemma}
\begin{proof}
If there exists a symbol $b$ in $\cA$ such $u_t = b$ for $t \in [1,k-n+1]$, then the conclusion holds trivially. 
By Lemma \ref{Lemma:Periodicity1}, if there is no such symbol, then there exists $m < n$ such that every word in $W'_m(u)$ has a unique extension in $W'_{m+1}(u)$. Let $i \in [1, k-2n]$, and let $w = u[i,i+n-1]$. Write $w = w_1 w_2$, where $|w_2| = m$. Note that $w_2$ is in $W'_m(u)$ (since $i < k-2n$ and $w$ has length $n$), and therefore $w_2$ has a unique right extension in $W'_{m+1}(u)$. Hence, $w$ has a unique right extension in $W'_{n+1}(u)$. We have shown that $u[i+1,i+n]$ is determined uniquely by $u[i,i+n-1]$ for each $i$ in $[1,k-2n]$. Since $|W_n(u)|=j$, there exist $1 \leq t_1 < t_2 \leq j+1$ such that $u[t_1,t_1+n-1] = u[t_2,t_2+n-1]$. 
Since $u[i+1,i+n]$ is determined uniquely by $u[i,i+n-1]$ for each $i$ in $[1,k-2n]$, we now have that $u[t_1+\ell(t_2-t_1),t_1+\ell(t_2-t_1) +n - 1] = u[t_1,t_1 +n - 1]$ for all $\ell$ such that $t_1+\ell(t_2-t_1) \leq k-2n$. Thus, the word $u[n,k-n]$ is periodic with period not greater than $j$. 
\end{proof}

We now return to the general setting of patterns on subsets of $\Z^d$, with arbitrary $d \geq 1$. Suppose $u$ is in $\cA^{F_k}$. The following lemma shows that if $|W_n(u)|$ is small enough, then there exists a finite orbit $\gamma$ such that $\gamma$ appears in $X_{\omega}$ whenever $u$ does. In the proof, we show that if $|W_n(u)| \leq n/2$, then $u|_{\interior_n(F_k)}$ is a totally periodic pattern that can be extended to a totally periodic configuration on $\Z^d$ without adding any new $F_n$-patterns. 
\begin{lemma} \label{Lemma:Smallj}
Suppose $k>4n$, $u \in \cA^{F_k}$, and $|W_n(u)|\leq n/2$. Then there exists a finite orbit $\gamma$ such that $W_n(u) \supset W_n(\gamma)$ and $|\gamma|\leq n/2$. 
\end{lemma}
\begin{proof}
Let $k > 4n$ and $u \in \cA^{F_k}$ with $|W_n(u)| \leq n/2$. Recall that $\pi_i$ is the projection of $\Z^d$ onto the hyperplane passing through $\1$ perpendicular to the standard basis vector $e_i$. For each $i \in [1,d]$, we define $f_{i} : [0,k-1] \to \cA^{\pi_i(F_n)}$ by 
\begin{equation} \label{Eqn:Definef_j}
f_{i}(m) = u|_{\pi_i(F_n) + m e_i}.
\end{equation}
 Viewing $f_{i}$ as a word of length $k$ (with alphabet $\cA^{\pi_i(F_n)}$ and $d = 1$), we see that each word of length $n$ in $f_i$ corresponds to an $F_n$-pattern in $u$. Hence $|W_n(f_{i})| \leq |W_n(u)|\leq n/2$. Then by Lemma \ref{Lemma:Periodicity2}, we obtain that $f_{i}|_{[n,k-n]}$ is periodic with period not greater than $n/2$. 

For each $i$ in $[1,d]$, define $r_i$ as the least period of $f_{i}|_{[n,k-n]}$. Let $L$ be the lattice in $\Z^d$ generated by $\{r_1e_1,\dots,r_de_d\}$. 
Let us now show that if $p \in L$ and $F_n + (n-1)\1+p \subset \interior_n(F_k)$, then  
\begin{equation} \label{Eqn:PeriodicityFiniteWords}
u|_{F_n+(n-1)\1} = u|_{F_n + (n-1)\1+p}.
\end{equation}

Fix such a $p$, and note that $r_i$ divides $p_i$ for each $i = 1, \dots, d$. For each $j$ in $[1,d]$, let $q_j = \sum_{i=1}^j p_i e_i$, so that $p = q_d$. Let $q_0 = 0 \in \Z^d$. We claim by induction on $j$ that $u|_{F_n+(n-1)\1} = u|_{F_n + (n-1)\1+q_j}$ for each $j \in [0,d]$. The base case ($j = 0$) is trivial. Now suppose for induction that $u|_{F_n+(n-1)\1} = u|_{F_n + (n-1)\1+q_j}$ holds for some $j < d$. Define $g : [0,k-1] \to \cA^{\pi_{j+1}(F_n)}$  by
\begin{align*}
 g(m) = u|_{\pi_{j+1}(F_n) +(n-1)\1 + q_j + m e_{j+1}}.
\end{align*}
Viewing $g$ as a word of length $k$, we see that $|W_n(g)| \leq |W_n(u)|\leq n/2$. Then by Lemma \ref{Lemma:Periodicity2}, we obtain that $g|_{[n,k-n]}$ is periodic with period not greater than $n/2$. Furthermore, the inductive hypothesis gives that $g|_{[n,2n-1]} = f_{j+1}|_{[n,2n-1]}$ (recall that $f_{j+1}$ was defined in (\ref{Eqn:Definef_j}) above).  Then by the Fine-Wilf Theorem \cite{FineWilf}, we conclude that $g|_{[n,k-n]} = f_{j+1}|_{[n,k-n]}$. Since $r_{j+1}$ is a period for $f_{j+1}|_{[n,k-n]}$ (by definition of $r_{j+1}$) and $g|_{[n,k-n]} = f_{j+1}|_{[n,k-n]}$, we see that $r_{j+1}$ is a period for $g|_{[n,k-n]}$. Since $r_{j+1}$ divides $p_{j+1}$ (since $p$  is in the lattice $L$), we conclude that
\begin{equation} \label{Eqn:NCAA}
u|_{F_n + (n-1)\1+q_{j+1}} = u|_{F_n+(n-1)\1+q_j}.
\end{equation}
Combining the inductive hypothesis with (\ref{Eqn:NCAA}), we obtain
\begin{equation*}
u|_{F_n + (n-1)\1+q_{j+1}} = u|_{F_n+(n-1)\1+q_j} = u|_{F_n + (n-1)\1}.
\end{equation*}
which concludes the inductive step. Hence $u|_{F_n+(n-1)\1} = u|_{F_n + (n-1)\1+p}$, and we have verified (\ref{Eqn:PeriodicityFiniteWords}). 

With (\ref{Eqn:PeriodicityFiniteWords}) established, let us now construct the finite orbit in the conclusion of the lemma.
Let $D = n\1 + \prod_{i=1}^d [0,r_i-1]$. Let $x$ be the point in $\cA^{\Z^d}$ defined by $x|_D = u|_D$ and $x(v+r_i e_i) = x(v)$ for all $v$ in $\Z^d$. Let $\gamma$ be the finite orbit containing $x$. Note that $D$ is a fundamental domain for $\gamma$ (meaning that every point in $\gamma$ is uniquely determined by its restriction to $D$). Then by (\ref{Eqn:PeriodicityFiniteWords}) and the fact that $k > 4n$, we have that 
\begin{equation} \label{Eqn:WordsInX}
 \{ x|_{F_n+v-\1} : v \in D\} = \{ u|_{F_n+v-\1} : v \in D\}.
\end{equation}
By (\ref{Eqn:WordsInX}), we obtain
\begin{align*}
 W_n(\gamma) & = W_n(x) \\
  & = \{ x|_{F_n+v-\1} : v \in D\} \\
  & = \{u|_{F_n+v-\1} : v \in D\} \\
  & \subset W_n(u).
\end{align*}
Furthermore, since $D \subset F_n+(n-1)\1$ and any point in $\gamma$ is determined by its restriction to $D$, we have that $|\gamma| = |W_n(x)|$. Therefore we have that $|\gamma| = |W_n(x)| \leq |W_n(u)| \leq n/2$, which concludes the proof. 
\end{proof}


\subsection{Basics of repeat covers} \label{Sect:BasicsOfRepeats}

In the following lemmas, we establish some basic results relevant to repeat covers. (Recall that repeat covers and associated notation are defined in Section \ref{Sect:Repeats}.) These results are used in Section \ref{Sect:EmptinessProof}. First, we show that a pattern $u$ with shape $E$ is uniquely characterized by a repeat cover $J$ of $u$ and the pattern $u|_{E \setminus \rArea{J}}$. 
\begin{lemma} \label{Lemma:DefineAWord}
 Suppose $J \subset \cC_n(E) \times \cC_n(E)$ and $w \in \mathcal{A}^{E \setminus \rArea{J}}$. Then there is at most one pattern $u$ in $\mathcal{A}^{E}$ such that $u|_{E \setminus \rArea{J}} = w$ and $J$ is a repeat cover for $u$.
\end{lemma}
\begin{proof}
Let $J \subset \cC_n(E) \times \cC_n(E)$ and $w \in \mathcal{A}^{E \setminus \rArea{J}}$ be given, and suppose $u,v \in \mathcal{A}^{E}$ satisfy
\renewcommand{\labelenumi}{(\alph{enumi})}
\begin{enumerate}
 \item $u|_{E \setminus \rArea{J}} = v|_{E \setminus \rArea{J}} = w$,
 \item $J$ is a repeat cover for $u$, and
 \item $J$ is a repeat cover for $v$.
\end{enumerate}
We will show that $u = v$. In fact, we will show that $u_t = v_t$ for each $t$ in $E$ by induction on $t$ (in the lexicographic ordering). Let $m = m(E)$ be the lexicographically minimal element of $E$. By the definition of repeat cover, $m$ is not in $\rArea{J}$. Therefore $u_{m} = v_{m} = w_{m}$ by (a). Now suppose for induction that $t$ is in $E$ and $u_s = v_s$ for all $s$ in $E$ with $s < t$. If $t$ is not in $\rArea{J}$, then $u_t = v_t = w_t$ by (a). Suppose that $t$ is in $\rArea{J}$. Then there exists $(S_1,S_2) \in J$ such that $t \in S_2$. Let $p = m(S_2) - m(S_1)$. Since $(S_1,S_2)$ is a repeat for both $u$ and $v$ (by (b) and (c)), we have that $u_t = u_{t-p}$ and $v_t = v_{t-p}$. Since $m(S_1) < m(S_2)$, we have that $t-p < t$, and therefore $u_{t-p} = v_{t-p}$ by the induction hypothesis. Hence, we have shown that $u_t = u_{t-p} = v_{t-p} = v_t$, which completes the proof.
\end{proof}


In the following lemma, for a pattern $u$ with shape $F_k$ and an $n$-repeat cover $J$ for $u$, we bound the cardinality of $F_k \setminus \rArea{J}$ in terms of the number of distinct $F_n$-patterns in $u$. The rough idea is that each $p$ in $F_k$ can be associated to a pattern in $W_n(u)$. The pattern associated to $p$ is either the lexicographically first occurrence of that pattern, in which case it contributes to $|W_n(u)|$, or it is a repeated $F_n$-pattern, in which case it contributes to $|\rArea{J}|$. This imprecise argument gives the idea that $|F_k| \approx |W_n(u)| + |\rArea{J}|$. The following lemma makes this idea precise by defining an injective map from a large subset of $F_k \setminus \rArea{J}$ into $W_n(u)$.
\begin{lemma} \label{Lemma:Nuggets}
Suppose $k > (2d+1)n$, $u \in \cA^{F_k}$ with $|W_n(u)|=j$, and $J$ is an $n$-repeat cover for $u$. Then
\begin{equation*}
 k^d - |\rArea{J}| \leq j\biggl(1+\frac{4 d n}{k}\biggr) .
\end{equation*}
\end{lemma}
\begin{proof}
Let $k >(2d+1)n$, and let $u$ be in $\cA^{F_k}$ with $|W_n(u)|=j$. Let $J$ be an $n$-repeat cover for $u$. For $i \in [1,d]$ and $\ell \in [1,k-n+1]$, define $H_i^{\ell}$ to be the width-$n$ hyperplane perpendicular to $e_i$ at position $\ell$:
\begin{equation*}
H_{i}^{\ell} = \{ v \in F_k : v_i \in [\ell, \ell+n-1] \}.
\end{equation*}
Note that $\bigcup_{\ell} (H_i^{\ell} \setminus \rArea{J}) = F_k \setminus \rArea{J}$, and for each point $p$ in $F_k$,
\begin{equation*}
 |\{ \ell : p \in H_i^{\ell} \}| \leq n.
\end{equation*}
Therefore (by Lemma \ref{Lemma:CombLemma}),
\begin{equation} \label{Eqn:SumHi}
 \sum_{\ell} |H_i^{\ell} \setminus \rArea{J}| \leq n |F_k \setminus \rArea{J}| = n (k^d - |\rArea{J}|).
\end{equation}
For each $i \in [1,d]$, choose $\ell(i)$ such that $|H_i^{\ell(i)} \setminus \rArea{J}| \leq |H_i^{\ell} \setminus \rArea{J}|$ for all $\ell$. 
Using that the minimum of a finite set of real numbers is less than or equal to the average and then applying (\ref{Eqn:SumHi}), we have that
\begin{equation} \label{Eqn:Hineq}
|H_i^{\ell(i)} \setminus \rArea{J}| \leq \frac{1}{k-n+1} \sum_{\ell} |H_i^{\ell} \setminus \rArea{J}| \leq  \frac{n(k^d-|\rArea{J}|)}{k-n+1}.
\end{equation}

For ease of notation, define the set
\begin{equation*}
T =   \rArea{J} \cup \biggl( \bigcup_i H_i^{\ell(i)} \biggr);
\end{equation*}
see Figure \ref{AunionHyperplanes}.
\begin{figure}
 \centering
\includegraphics[scale=.5]{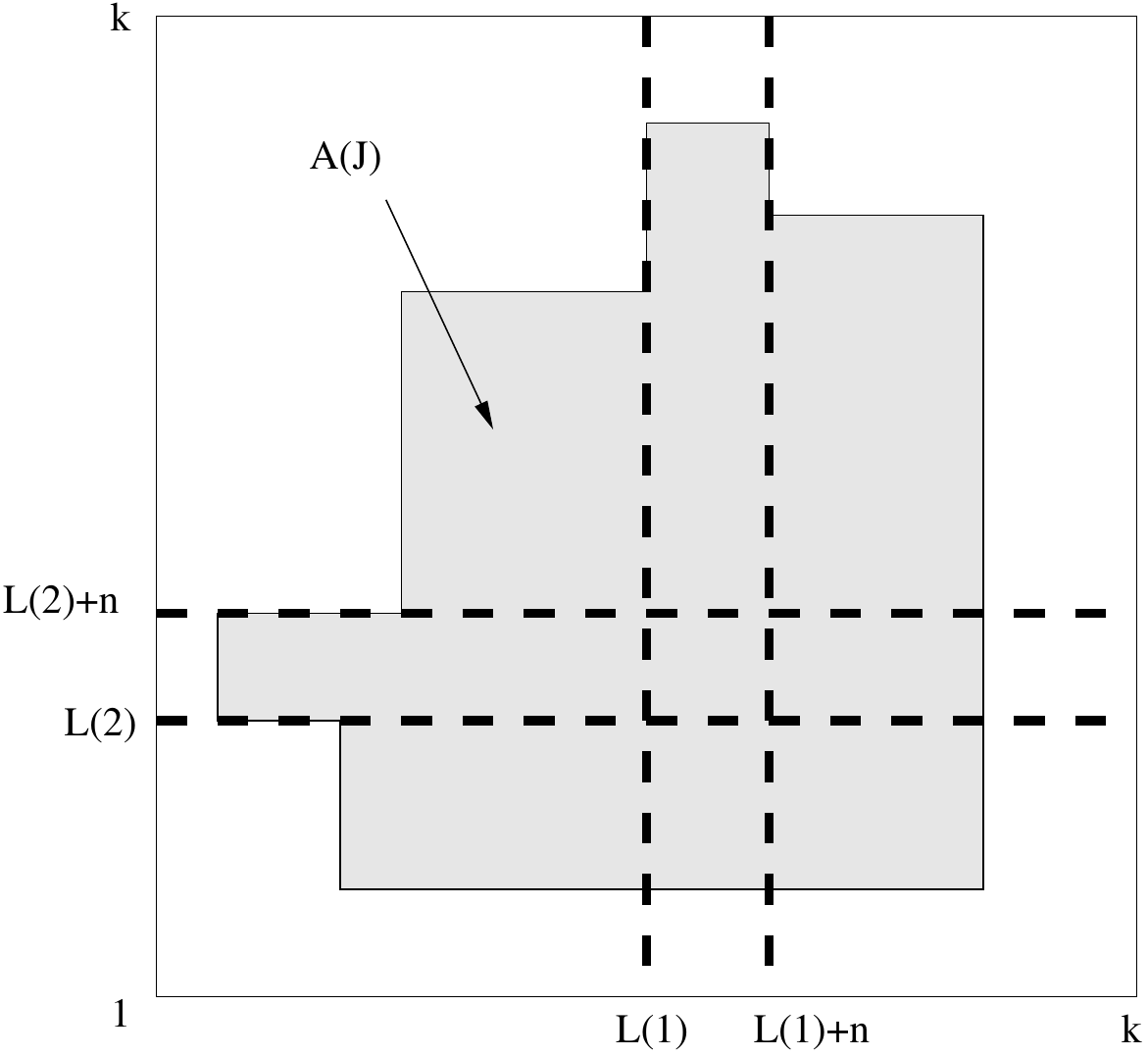}
\caption{An example of the set $T = A(J) \cup (H_1^{\ell(1)} \cup H_2^{\ell(2)})$. The set $A(J)$ appears shaded. Here $\ell(i)$ is labeled $L(i)$ for each $i = 1,2$, and the sets $H_i^{\ell(i)}$ are bounded by dashed lines.}
\label{AunionHyperplanes}
\end{figure}
We now claim that there is an injection from $F_k \setminus T$ into $W_n(u)$, and therefore 
\begin{equation} \label{Eqn:Jake}
\bigl|F_k \setminus T \bigr| \leq |W_n(u)| = j.
\end{equation}
 Let us prove the claim. For each $i \in [1,d]$, let $R_i^+$ ($R_i^-$) be the ``half-hypercube'' on the ``$+$'' (``$-$'') side of the thickened hyperplane $H_i^{\ell(i)}$, \textit{i.e.}, $R_i^{+} = \{ p \in F_k : p_i > \ell(i)+n-1 \}$ and $R_i^{-} = \{ p \in F_k  : p_i < \ell(i) \}$ (so that $F_k = R_i^{-} \sqcup H_i^{\ell(i)} \sqcup R_i^+$, as in Figure \ref{Hyperplane_pic}). 
For each $w$ in $\{+,-\}^d$, let $R(w)$ be the ``hyper-quadrant'' $R(w) = \bigcap_i R_i^{w_i}$. Also, for an $n$-cube $S$ in $\cC_n(F_k)$, let $S(w)$ be the corner in $S$ specified by $w$: if $S = F_n + v$, then $S(w)$ is the point whose $i$-th coordinate satisfies
\begin{equation*}
 S(w)_i = v_i + \left\{ \begin{array}{ll}
                       1, & \text{ if } w_i = - \\
                       n, & \text{ if } w_i = +.
                      \end{array}
             \right.
\end{equation*}

\begin{figure}
 \centering
 \includegraphics[scale=.5]{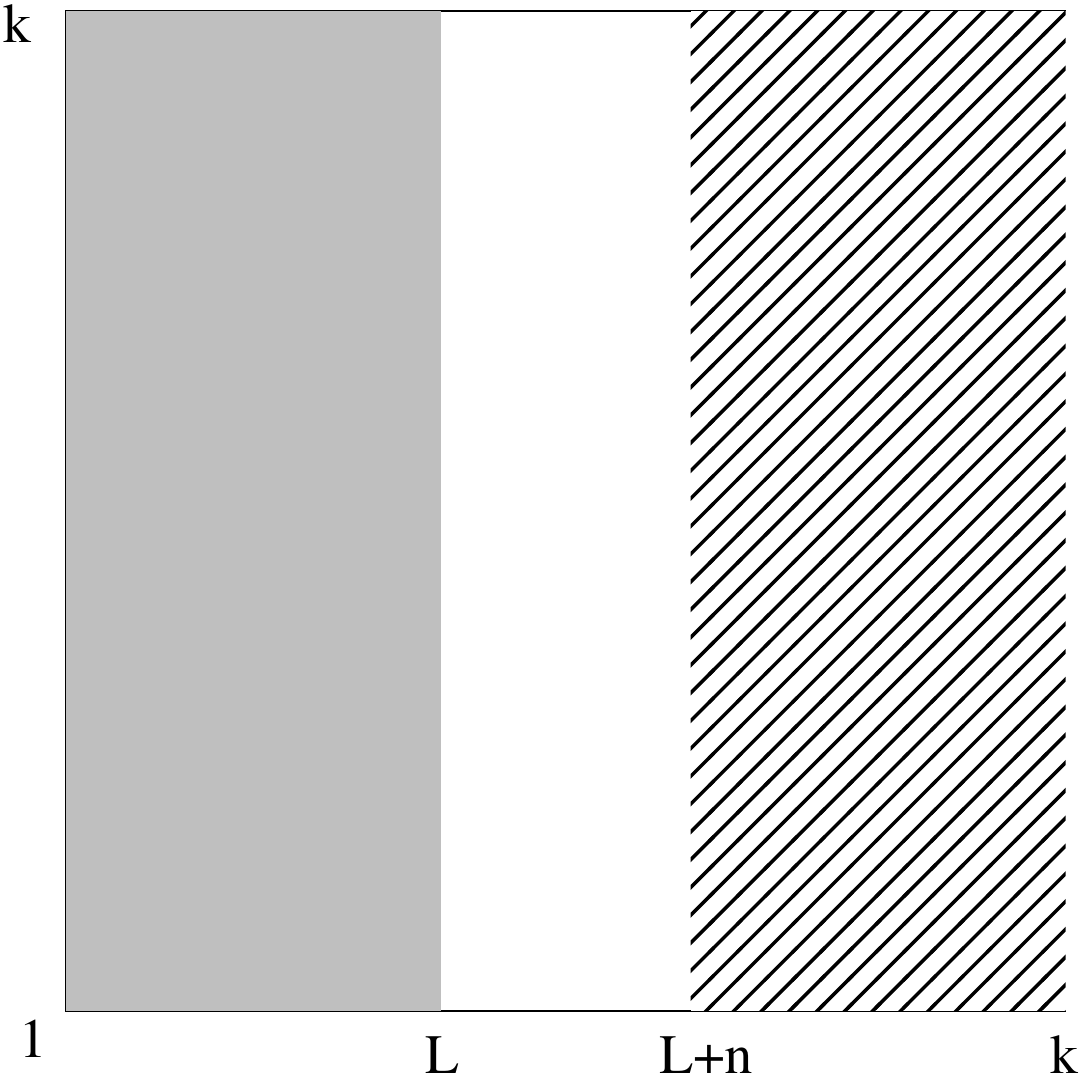}
 \caption{A decomposition of $F_k$ as $R_1^{-} \sqcup H_1^{\ell(1)} \sqcup R_1^{+}$; here $R_1^{-}$ appears shaded, $R_1^+$ appears with stripes, $H_1^{\ell(1)}$ appears in between, and $L = \ell(1)$.}
 \label{Hyperplane_pic}
\end{figure}

Now for $p$ in $R(w)$, let $S_p$ be the $n$-cube in $\cC_n(F_k)$ such that $S(w) = p$. Since each $H_i^{\ell(i)}$ has width $n$ (in the $i$-th direction), we have that the map $F_k \setminus (\cup_i H_i^{\ell(i)}) \to \cC_n(F_k)$ given by $p \mapsto S_p$ is an injection. Finally, we have that the map from $F_k \setminus T$ into $W_n(u)$ given by $p \mapsto u|_{S_p}$ is an injection. Indeed, if $p,q \notin T$ and $u|_{S_p} = u|_{S_q}$, then $S_p = S_q$ (since $J$ is a repeat cover and $p,q \notin \rArea{J}$), and therefore $p = q$ (since $p \mapsto S_p$ is injective).

By (\ref{Eqn:Jake}), union bound, and (\ref{Eqn:Hineq}) we have that
\begin{align*}
j \geq \bigl|F_k \setminus T \bigr| & = k^d - |\rArea{J}| - \biggl|\bigcup_i H_i^{\ell(i)} \setminus \rArea{J} \biggr| \\
& \geq k^d - |\rArea{J}| - \sum_i |H_i^{\ell(i)} \setminus \rArea{J}| \\
& \geq k^d - |\rArea{J}| - \frac{d n (k^d - |\rArea{J}|)}{k-n+1} \\
& = (k^d-|\rArea{J}|)\biggl( 1 - \frac{d n }{k-n+1} \biggr).
\end{align*}
Dividing by $1- \frac{dn}{k-n+1}$ gives 
\begin{equation} \label{Eqn:jIsBigIneq}
k^d-|\rArea{J}| \leq j \biggl( \frac{1}{1 - \frac{d n }{k-n+1} } \biggr).
\end{equation}
Since $k > (2d+1)n$, we have that $n -1 < k/2$ and $dn/(k-n+1) < 1/2$. Using these facts, along with (\ref{Eqn:jIsBigIneq}) and the elementary fact that if $0 \leq x \leq 1/2$, then $\frac{1}{1-x} \leq 1 + 2x$, we see that
\begin{align*}
k^d-|\rArea{J}| \leq j \biggl( \frac{1}{1 - \frac{d n }{k-n+1} } \biggr) \leq j \biggl( 1 + \frac{2 d n}{k-n+1} \biggr) \leq j \biggl( 1+ \frac{ 4 d n}{k} \biggr),
\end{align*}
as desired.
\end{proof}

\subsection{Repeat covering lemmas} \label{Sect:RepeatCovers}

The purpose of this section is to construct ``efficient'' repeat covers $J$, where ``efficient'' means that $|J|$ is small enough for our purposes. Consider $u$ in $\cA^{F_k}$ with $|W_n(u)|=j$. The case $j \leq n/2$ is handled by Lemma \ref{Lemma:Smallj}, and we will not deal with that case in this section. For $j \geq n/2$, we construct a repeat cover of $u$ by decomposing $F_k$ into three regions and covering the repeated $F_n$-patterns in each region separately. The decomposition of $F_k$ that we use depends on $j$ as follows. If $\ell \geq 1$ is an integer such that $n^{\ell}/5^{d} \leq j < n^{\ell+1}/3^{d}$, then we decompose $F_k$ into the following three regions: i) a (relatively) small neighborhood of the $\ell$-skeleton of $F_k$, ii) the $n$-neighborhood of the $\ell$-skeleton of $F_k$ (minus the first region), and iii) the complement of the first two regions. 
These three regions are handled separately with help from Lemmas \ref{Lemma:GeneralCoding}, \ref{Lemma:Place2}, and \ref{Lemma:CoverInterior}, respectively. In Lemma \ref{Lemma:ThreePieces}, we combine these lemmas to produce a repeat cover $J$ for $u$ with an upper bound on $|J|$. The culmination of this section is Lemma \ref{Lemma:FinalCoveringLemma}, which gives an asymptotic upper bound on $|J|$ under some conditions on the asymptotic relationship between the parameters involved in constructing these repeat covers. The purpose of all of this work is that Lemma \ref{Lemma:FinalCoveringLemma} plays an important role in the proof of Theorem \ref{Thm:EmptinessIntro} given in Section \ref{Sect:EmptinessProof}.

\subsubsection{Covering sets near a face} \label{Sect:CoverNearAFace}

The following lemma allows us to find efficient covers of sets near a face of dimension at least $1$. For example, see Figure \ref{NearAnEdge_pic}. We achieve this efficiency by reducing the problem to a one-dimensional covering problem and applying Lemma \ref{Lemma:IntervalCoveringLemma}.
\begin{lemma} \label{Lemma:GeneralCoding}
 Suppose $E = F_k(\mathcal{I},s)$ is a face of dimension $\ell \geq 1$ in $F_k$, and let $R = B(E,n) \cap F_k$. Further suppose that $\cC \subset \cC_n(F_k)$, and let $U = \cup_{S \in \cC} (S \cap R)$. Then there exists $\cC' \subset \cC$ such that $U = \cup_{S \in \cC'} (S\cap R)$ and $|\cC'|\leq 2 |U|/n$.
\end{lemma}
\begin{figure}
 \centering
 \includegraphics[scale=.5]{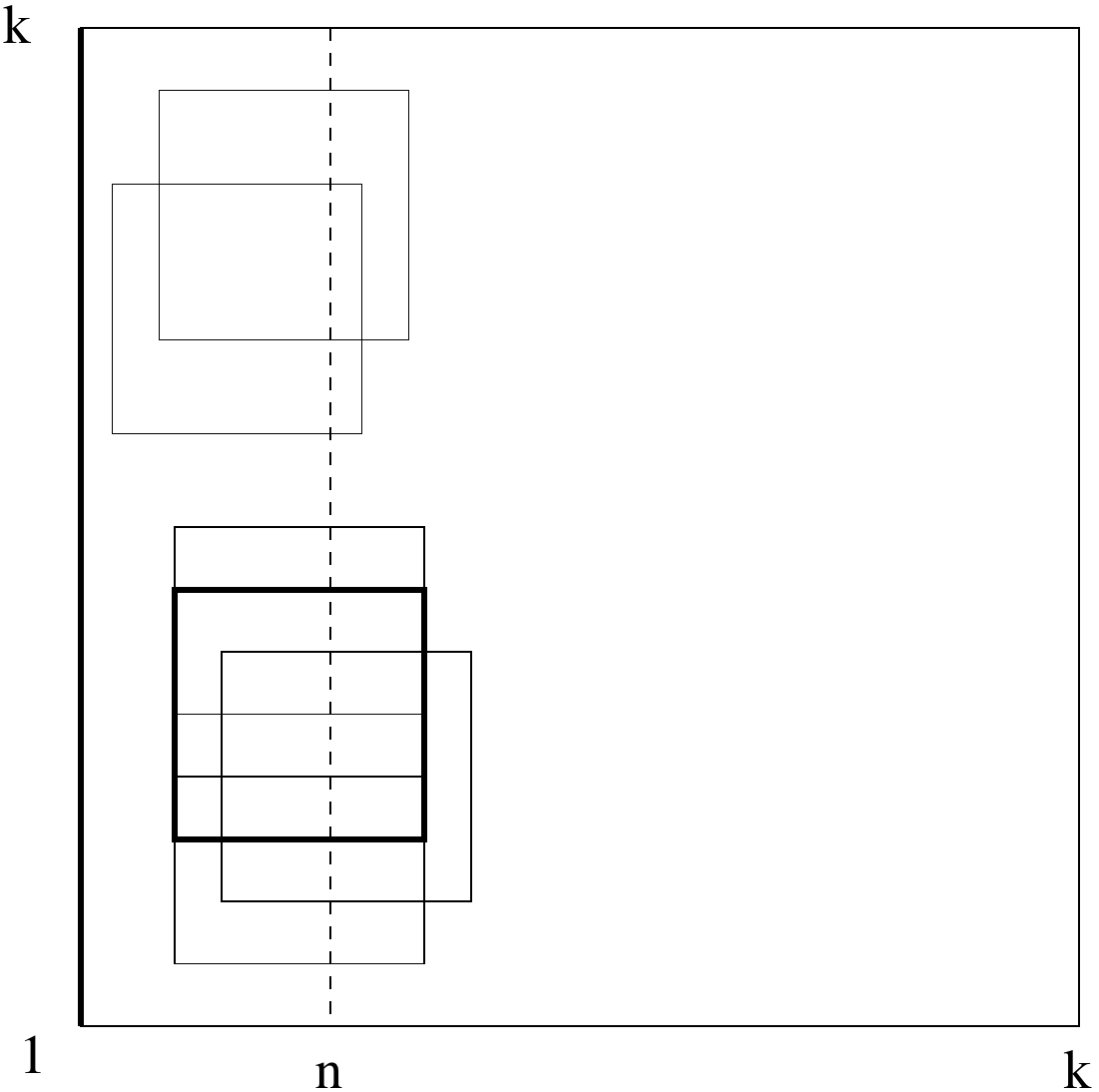}
 \caption{A union of $n$-cubes in $F_k$, near the face $E = F_k(\mathcal{I},s)$, which is chosen to be the left edge. Notice that the bold $n$-cube may be removed without changing the union of the collection.}
 \label{NearAnEdge_pic}
\end{figure}
\begin{proof}
 Let $E$, $R$, $\cC$ and $U$ be as above. Since $E$ has dimension at least $1$, we may assume without loss of generality that $1 \notin \mathcal{I}$. (If $\mathcal{I}$ does not satisfy this condition, then we may rotate $F_k$ so that it does.) Let $L_n$ be the line segment of $F_n$ in direction $e_1$ passing through $\1$, \textit{i.e.}, $L_n = \{ \1 + me_1 : m \in [0,n-1]\}$. For each $S$ in $\cC$, we define the line segment $\phi(S) = L_n + v$, where $v$ is the unique vector satisfying $S = F_n + v$.

 Let $L$ be a line in direction $e_1$ that intersects $R$, and let 
\begin{equation*}
 F =  \bigcup_{S \in \cC} \phi(S) \cap L.
\end{equation*}
By definition, $F$ is contained in the line $L$, and $\{\phi(S) \cap F : S \in \cC\}$ is a cover of $F$ by intervals. By Lemma \ref{Lemma:IntervalCoveringLemma}, we obtain that there is a subset $\cC(L)$ of $\cC$ such that 
\begin{itemize}
 \item $F \subset \bigcup_{S \in \cC(L)} \phi(S)$,
 \item $\phi(S) \subset L$ for each $S \in \cC(L)$, and
 \item for each point $p$ in $F$, we have $|\{ S \in \cC(L) : p \in \phi(S) \}| \leq 2$.
\end{itemize}
Now let 
\begin{equation*}
 \cC' = \bigcup_{L} \cC(L).
\end{equation*}
where the union is over all lines $L$ in direction $e_1$ such that $L \cap R \neq \varnothing$.

As $\cC' \subset \cC$, we have $\bigcup_{S \in \cC'} S \cap R \subset U$. To show the reverse containment, let $x \in U$. By definition of $U$, there exists $S$ in $\cC$ such that $x \in S \cap R$. Let $L$ be the line in direction $e_1$ containing $\phi(S)$. By construction of $\cC(L)$, we have that $\phi(S) \subset \cup_{S' \in \cC(L)} \phi(S')$, and therefore $S \subset \cup_{S' \in \cC(L)} S'$. Hence there exists $S'$ in $\cC'$ such that $x \in S'$. Since $x \in U$ was arbitrary, we conclude that
\begin{equation*}
 U = \bigcup_{S \in \cC'} S \cap R.
\end{equation*}
Furthermore, for each $p$ in $U$, we have that $|\{S \in \cC' : p \in \phi(S)\}| \leq 2$. Therefore (using Lemma \ref{Lemma:CombLemma})
\begin{equation*}
 |\cC'| n =  \sum_{S \in \cC'} |\phi(S)| \leq 2 \biggl| \bigcup_{S \in \cC'} \phi(S) \biggr| \leq 2 |U|.
\end{equation*}
Dividing by $n$, we obtain that $|\cC'| \leq 2 |U|/n$, as desired.
\end{proof}

\subsubsection{Covering regions between faces and interiors} \label{Sect:CoverBetween}

The goal of this subsection is to prove Lemma \ref{Lemma:Place2}, which will be used in Lemma \ref{Lemma:ThreePieces} to construct efficient repeat covers in regions of the form $(F_k \cap B(F_{k,\ell},n)) \setminus B(F_{k,\ell},r)$, with $r < n$. In order to do so, we will require some additional terminology.

Let $E = F_k(\mathcal{I},s)$ be a face of $F_k$ (with notation as in Section \ref{Sect:Rectangles}). For $p$ in $F_k$ and $i$ in $\mathcal{I}$, let $\pi_{s,i}(p)$ be the point in $F_k$ defined by
\begin{equation*}
 \pi_{s,i}(p)_t = \left\{ \begin{array}{ll}
                       p_t, & \text{ for } t \neq i \\
                       s(i), & \text{ for } t=i.
                      \end{array}
              \right.
\end{equation*}
Note that a single $\pi_{s,i}$ does not necessarily project all points to $E$, because it only projects along a single direction (see Figure \ref{Projections_pic}). Let $\Line(p,q)$ denote the line segment in $F_k$ from $p$ to $q$. 
\begin{figure}
 \centering
 \includegraphics[scale=.5]{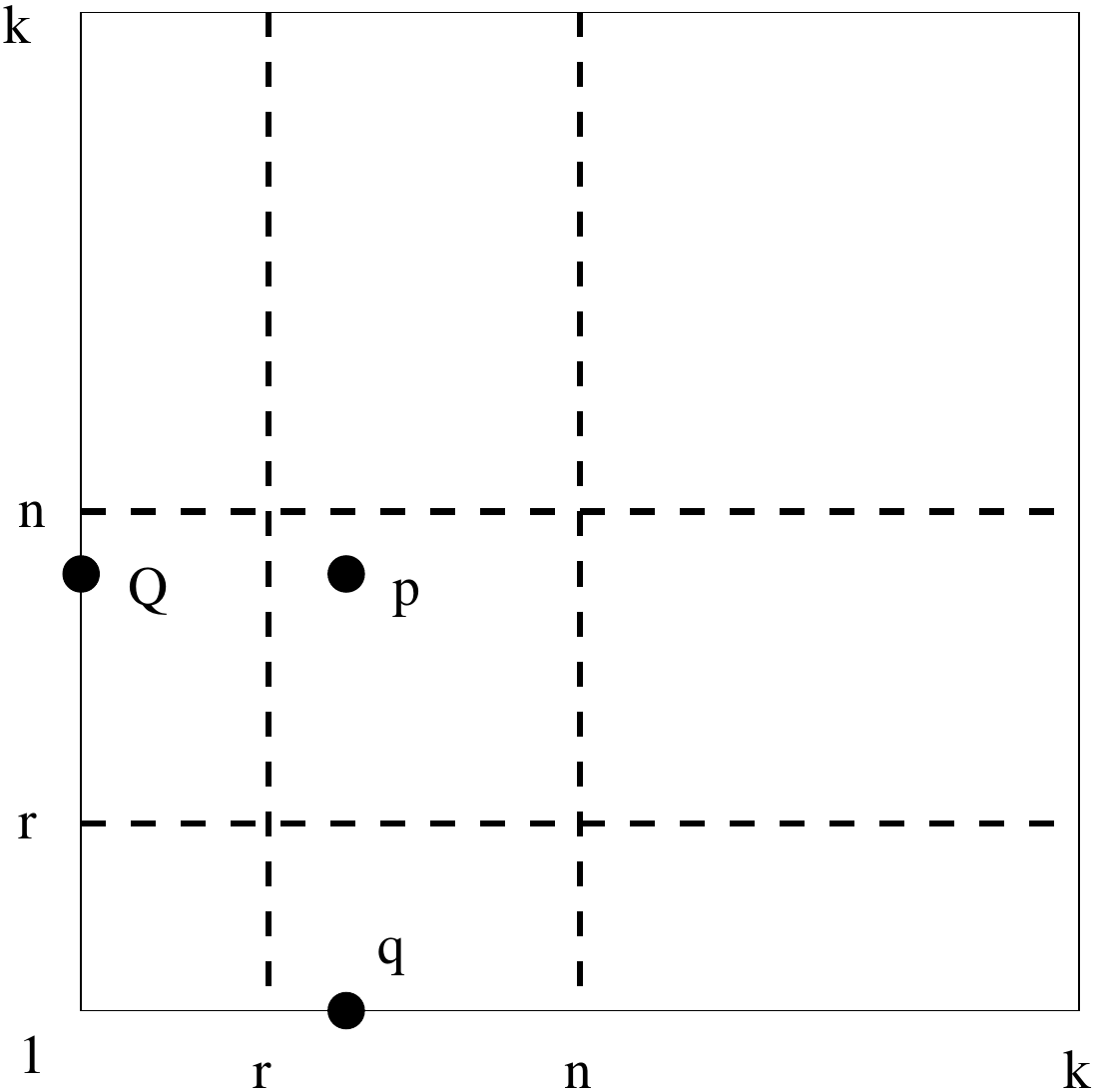}
 \caption{Here we choose the face $E = F_k(\mathcal{I},s)$, with $\mathcal{I} = \{1,2\}$ and $s \equiv 1$, \textit{i.e.} $E$ consists of the point $\1$. For the indicated point $p$ (which lies in $( F_k \cap B(E,n) ) \setminus B(E,r)$), we have $\pi_{s,1}(p) = Q$ and $\pi_{s,2}(p) = q$.}
 \label{Projections_pic}
\end{figure}

Given a set $T \subset F_k$ and a face $E = F_k(\mathcal{I},s)$, we say that $p \in T$ is \textit{$(E,T)$-necessary} if for each $i \in \mathcal{I}$, we have that $\Line(p,\pi_{s,i}(p)) \cap T= \{p\}$. For example, see Figure \ref{NecessaryPts_pic}.
\begin{figure}
 \centering
 \includegraphics[scale=.5]{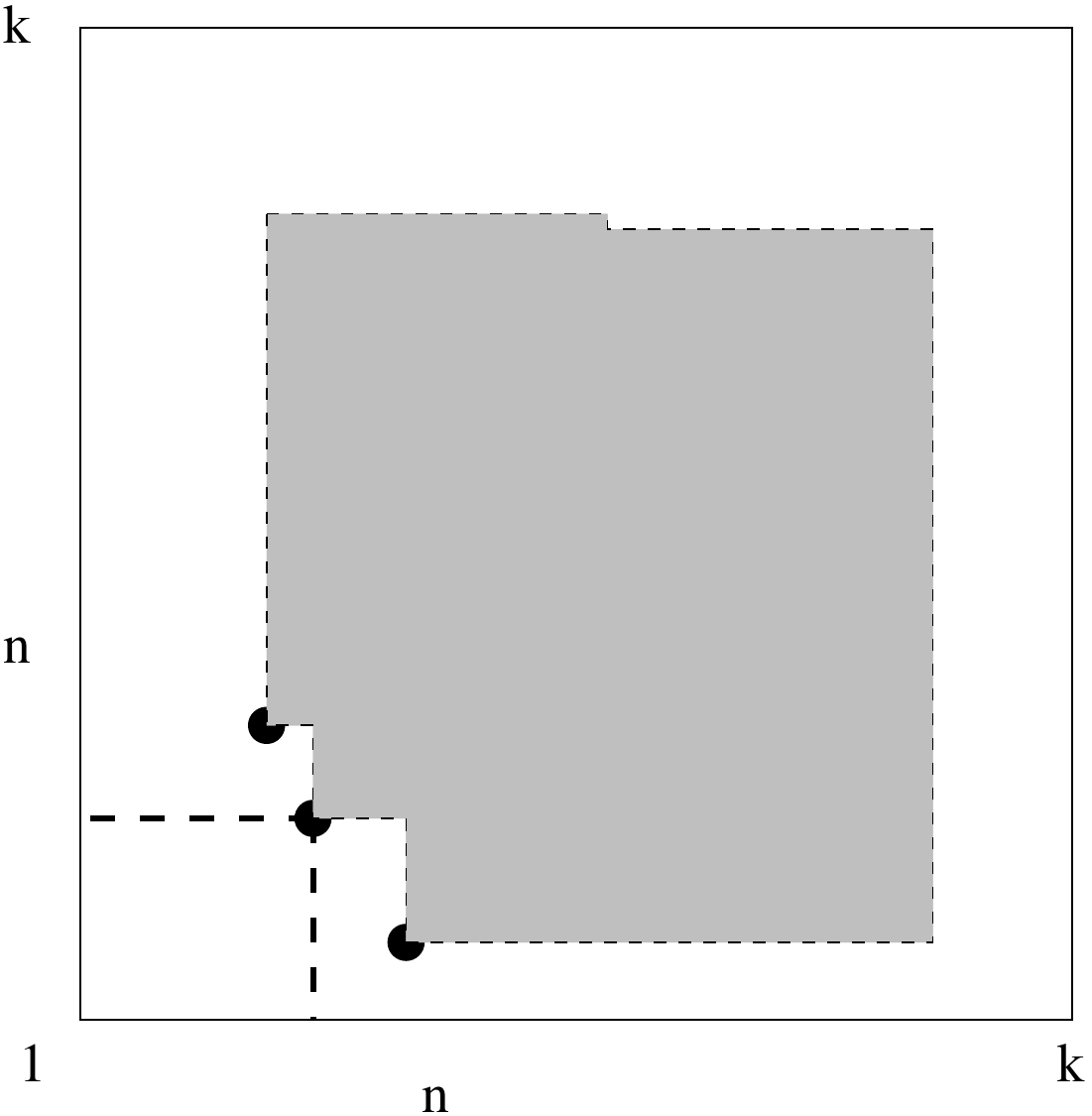}
 \caption{Here we choose the face $E = F_k(\mathcal{I},s)$, with $\mathcal{I} = \{1,2\}$ and $s \equiv 1$, \textit{i.e.} $E$ consists of the point $\1$. If $T$ is the shaded region, then there are three $(E,T)$-necessary points, marked with bold dots. The dashed lines demonstrate why the middle point is an $(E,T)$-necessary point.}
 \label{NecessaryPts_pic}
\end{figure}

 Given a set $T \subset F_k$ and $\ell \in [0,d]$, we say that $p \in T$ is $(\ell,T)$-necessary if $p$ is $(E,T)$-necessary for some face $E$ of dimension $\ell$. 
Observe that if $p$ is an $(\ell,T)$-necessary point in $F_k \setminus B(F_{k,\ell},r)$, then there exist a face $F_k(\mathcal{I},s)$ of dimension $\ell$ and $i \in \mathcal{I}$ such that $\Line(p,\pi_{s,i}(p)) \cap T = \{p\}$ and $|\Line(p,\pi_{s,i}(p)) \setminus \{p\}| \geq r$.

The following lemma bounds from above the number of $(\ell,T)$-necessary points contained in $(F_k \cap B(F_{k,\ell},n)) \setminus B(F_{k,\ell},r)$. This lemma is used in Lemma \ref{Lemma:ThreePieces} to find efficient covers of $\rArea{J}$ in the region $(F_k \cap B(F_{k,\ell},n) ) \setminus B(F_{k,\ell},r)$. 
\begin{lemma} \label{Lemma:Place2}
 Suppose $n < k$ and $T \subset F_k$. Then for any $\ell \in [0,d-1]$ and $r \in [1,n)$, the number of $(\ell,T)$-necessary points in $(F_k \cap B(F_{k,\ell},n)) \setminus B(F_{k,\ell},r)$ is less than $d (k^d - |T|)/r $. 
\end{lemma}
\begin{proof}
 Let $n$, $k$, $T,\ell$ and $r$ be as above. Let $\mathcal{N}$ be the set of $(\ell,T)$-necessary points in $(F_k \cap B(F_{k,\ell},n) ) \setminus B(F_{k,\ell},r)$. If $p$ is in $\mathcal{N}$, then there exist a face $F_k(\mathcal{I},s)$ of dimension $\ell$ and $i \in \mathcal{I}$ such that $\Line(p,\pi_{s,i}(p)) \cap T = \{p\}$ and $|\Line(p,\pi_{s,i}(p)) \setminus \{p\}| \geq r$. Arbitrarily choosing such a pair $(F_k(\mathcal{I},s),i)$ for each $(\ell,T)$-necessary point $p$, we let $L_p = \Line(p,\pi_{s,i}(p)) \setminus \{p\}$. Hence $L_p \subset F_k \setminus T$ and $|L_p| \geq r$. Note that for $q$ in $F_k$,
\begin{equation*}
 |p \in \mathcal{N} : q \in L_p| \leq d,
\end{equation*}
by the definition of necessary points and the fact that there are only $d$ cardinal directions in $\Z^d$. Thus (by Lemma \ref{Lemma:CombLemma}), we have that
\begin{equation*}
 |\mathcal{N}| r \leq \sum_{p \in \mathcal{N}} |L_p| \leq d |F_k \setminus T| = d( k^d - |T|),
\end{equation*}
which shows that $|\mathcal{N}| \leq d(k^d - |T|)/r$.
\end{proof}
%

\subsubsection{Covering the interiors of hypercubes} \label{Sect:CoverInterior}

In the following lemma, we produce an efficient sub-cover of the interior of a hypercube under the condition that the interior is ``covered densely," in some sense.
\begin{lemma} \label{Lemma:CoverInterior}
Let $F \subset \Z^{d_0}$ be a hypercube with side-length $k$ in dimension $d_0$. Suppose $\cC \subset \cC_n(F)$ has the property that for each $p$ in $\interior_n(F)$, there exists $S \in \cC$ such that if $c$ is the center of $S$, then $\rho(p,c) \leq n/6$. Then there exists $\cC' \subset \cC$ such that
\begin{enumerate}
 \item $\interior_n(F) \subset \bigcup_{S \in \cC'} S$, and
 \item $|\cC'| \leq (2k/n)^{d_0}$.
\end{enumerate}
\end{lemma}
\begin{proof}
Let $F$ and $\cC$ be as above. Fix a set $\mathcal{P} \subset \interior_n(F)$ that is $n/2$-separated and whose $n/3$-thickening covers $\interior_n(F)$, \textit{i.e.},
\renewcommand{\theenumi}{\roman{enumi}}
\renewcommand{\labelenumi}{(\theenumi)}
\begin{enumerate}
 \item if $p,q \in \mathcal{P}$ and $\rho(p,q) \leq n/2$, then $p = q$;
 \item $\interior_n(F) \subset \bigcup_{p \in \mathcal{P}} B(p,n/3)$.
\end{enumerate}
(Note that such a set always exists. For $d_0 = 1$, let $k - 2n = m (2n/3) + r$, with $r < 2n/3$. Define points $x_i^s = n+  i (2n/3)+s$, for $i = 1, \dots, m$. Then there exists some $s < n/3$ such that $\{x_i^s\}_i$ satisfies conditions (i) and (ii). For $d_0 > 1$, take the $d_0$-fold product of the set constructed for $d_0=1$.)

For each $p$ in $\mathcal{P}$, arbitrarily choose an $n$-cube $S(p)$ in $\cC$ such that $\rho(p,c) \leq n/6$, where $c$ is the center of $S(p)$ (and note that it is always possible to make such a choice by the hypothesis of the lemma). Let $\cC' = \{S(p) : p \in \mathcal{P} \}$. We claim that $\cC'$ satisfies the conclusions of the lemma.

To verify (1), let $q$ be in $\interior_n(F)$. Then by (ii), there exists $p$ in $\mathcal{P}$ such that $\rho(p,q) \leq n/3$. Let $c$ be the center of $S(p)$. Then we have
\begin{equation*}
 \rho(q,c) \leq \rho(p,q) + \rho(p,c) \leq \frac{n}{3} + \frac{n}{6} = \frac{n}{2}.
\end{equation*}
Hence $q$ is in $B(c,n/2) = S(p)$. Since $q$ was arbitrary, we have verified (1).

Let us now verify (2). First, we show that the map $p \mapsto S(p)$ is injective on $\mathcal{P}$. Indeed, suppose $S(p) = S(q)$ for $p,q$ in $\mathcal{P}$, and let $c$ be the center of $S(p)$. Then by definition of $S(p)$ and $S(q)$, we have $\rho(p,c) \leq n/6$ and $\rho(q,c) \leq n/6$. Hence
\begin{equation*}
 \rho(p,q) \leq  \rho(p,c) + \rho(q,c) = \frac{n}{6} + \frac{n}{6} = \frac{n}{3}.
\end{equation*}
Now by (i), we conclude that $p = q$, and therefore $p \mapsto S(p)$ is injective on $\mathcal{P}$. Thus $|\cC'| = |\mathcal{P}|$.

We now bound $|\mathcal{P}|$. By (i) and the triangle inequality, we have that if $p,q \in \mathcal{P}$ and $p \neq q$, then $B(p,n/4) \cap B(q,n/4) = \varnothing$. Hence
\begin{equation*}
 |\mathcal{P}| (n/2)^{d_0} = \sum_{p \in \mathcal{P}} |B(p,n/4)| = \biggl| \bigcup_{p \in \mathcal{P}} B(p,n/4) \biggr| \leq |F| = k^{d_0}.
\end{equation*}
Thus, we have that $|\cC'| = |\mathcal{P}| \leq (2k/n)^{d_0}$, as was to be shown.

\end{proof}

\subsubsection{Finding efficient repeat covers} \label{Sect:EfficientCovers}

In Lemma \ref{Lemma:ThreePieces}, we combine our results from Sections \ref{Sect:CoverNearAFace},  \ref{Sect:CoverBetween}, and  \ref{Sect:CoverInterior} to construct ``efficient'' repeat covers. Such repeat covers are built by decomposing the hypercube $F_k$ into three pieces and covering each piece separately. Each one of Lemmas \ref{Lemma:GeneralCoding}, \ref{Lemma:Place2}, and  \ref{Lemma:CoverInterior}  is used to bound the number of repeats needed to cover one of these pieces. Note that in our application of Lemma \ref{Lemma:ThreePieces}, we will choose $r$ and $\ell$ depending on $j$ and $n$.

\begin{lemma} \label{Lemma:ThreePieces}
Suppose $r \in [1,n) $, $\ell \in [1,d-1]$, $j < n^{\ell+1}/3^d$, and $u$ is in $\cA^{F_k}$ with $|W_n(u)|=j$. Then there exists an $n$-repeat cover $J$ of $u$ such that
\begin{align*}
|J| & \leq 2 c_{d,\ell} \frac{k^{\ell} r^{d-\ell}}{n} + \frac{d(k^d-|\rArea{J}|)}{r} + \sum_{d_0 = \ell+1}^d c_{d,d_0} \biggl( \frac{2k}{n} \biggr)^{d_0}, 
\end{align*}
(recall that $c_{d,\ell}$ is the number of faces of $F_k$ of dimension $\ell$).
\end{lemma}
\begin{proof}
Let $r,\ell,j$, and $u$ be as above. Let $J'$ be an $n$-repeat cover of $u$. We consider $F_k$ as a union of three regions, $F_k = R_1 \cup R_2 \cup R_3$, defined below. We will construct a repeat cover $J$ of $u$ by selecting repeats from $J'$ to cover $\rArea{J'}$ in each of these regions separately.

Let $R_1$ be the $r$-thickening of the $\ell$-skeleton in $F_k$: $R_1 = B(F_{k,\ell},r) \cap F_k$. Let $R_2$ be the $n$-thickening of the $\ell$-skeleton in $F_k$ minus $R_1$, \textit{i.e.} $R_2 = \bigl(F_k \cap B(F_{k,\ell},n) \bigr) \setminus B(F_{k,\ell},r)$. Lastly, let $R_3 = F_k \setminus (R_1 \cup R_2) = F_k \setminus B(F_{k,\ell},n)$. In light of (\ref{Eqn:ellSkeletonDecomp}), we have that 
\begin{equation} \label{Eqn:R3decomp}
 R_3 = F_k \setminus B(F_{k,\ell},n) =  \bigsqcup_{\substack{E \text{ face of } F_k \\ \dim(E) \geq \ell+1}} T_n(E),
\end{equation}
where $T_n(E)$ denotes the ``$n$-thickened interior'' of a face $E$ (defined in Section \ref{Sect:Rectangles}).

Let us now select repeats from $J'$ that cover $R_1 \cap \rArea{J'}$. Note that $R_1$ is the union of all the sets $B(E,r) \cap F_k$, where $E$ is a face of $F_k$ of dimension $\ell$. Let $\cC = \{S_2 : (S_1,S_2) \in J'\}$. For each face $E$ of dimension $\ell$, we apply Lemma \ref{Lemma:GeneralCoding} and conclude that there exists a subset $J_1(E) \subset J'$ such that 
\begin{align}
 \begin{split}
 \label{Eqn:Thing1}
B(E,r) \cap \rArea{J'} & = \bigcup_{(S_1,S_2) \in J_1(E)} S_2 \cap B(E,r) \\ & = B(E,r) \cap  \rArea{J_1(E)},
\end{split}
\end{align}
and 
\begin{equation} \label{Eqn:Thing2}
|J_1(E)| \leq \frac{2|B(E,r) \cap \rArea{J'}|}{n} \leq \frac{2 k^{\ell} r^{d-\ell}}{n}. 
\end{equation}
Let $J_1 = \cup_E J_1(E)$, where the union runs over all faces of $F_k$ of dimension $\ell$. By (\ref{Eqn:Thing1}) and (\ref{Eqn:Thing2}), we have
\begin{equation} \label{Eqn:R1covered}
 R_1 \cap \rArea{J'} = R_1 \cap \rArea{J_1},
\end{equation}
and
\begin{equation} \label{Eqn:J1bound}
 |J_1| \leq 2 c_{d,\ell} \frac{k^{\ell} r^{d-\ell}}{n}.
\end{equation}

Let us proceed to select repeats from $J'$ that cover $R_2 \cap \rArea{J'}$. Let $T = \rArea{J'}$. Let $\mathcal{N}$ be the set of $(\ell,T)$-necessary points in $R_2$. For each $p$ in $\mathcal{N}$, arbitrarily choose a repeat $(S_1(p),S_2(p))$ in $J'$ such that $p \in S_2(p)$, and let $J_2 = \{ (S_1(p),S_2(p)) : p \in \mathcal{N}\}$. We claim that 
\begin{equation} \label{Eqn:R2covered}
 R_2 \cap \rArea{J'} = R_2 \cap \Bigl( \rArea{J_1} \cup \rArea{J_2} \Bigr).
\end{equation}
By definition, the set on the right-hand side is contained in the set on the left-hand side. Let $q$ be in $R_2 \cap \rArea{J'}$. If $q$ is in $\mathcal{N}$, then by construction there exists a repeat $(S_1,S_2)$ in $J_2$ such that $q \in S_2$. Suppose $q$ is not in $\mathcal{N}$, \textit{i.e.}, $q$ is not $(\ell,T)$-necessary. Since $q$ is in $B(F_{k,\ell},n)$, there exists a face $E = F_k(\mathcal{I},s)$ of dimension $\ell$ such that $q \in B(E,n)$. Let $D$ be the region in $F_k$ ``between'' $q$ and $E$ (consisting of all vertices along shortest paths from $q$ to $E$). Then there exists an $(E,T)$-necessary point $q'$ in $D$ (otherwise $q$ would be $(E,T)$-necessary). If $q'$ is in $B(E,r)$, then there exists a repeat $(S_1,S_2)$ in $J_1$ such that $q' \in S_2$, and therefore $q \in S_2$ (since $q$ is in $B(E,n)$ and $q'$ is ``between'' $q$ and $E$, any translate of $F_n$ inside $F_k$ that contains $q'$ must also contain $q$). Otherwise, if $q'$ is not in $B(E,r)$, then $q'$ is in $\mathcal{N}$, and therefore $q \in S_2(q')$ (again, since $q$ is in $B(E,n)$ and $q'$ is ``between'' $q$ and $E$, any translate of $F_n$ inside $F_k$ that contains $q'$ must also contain $q$). In either case, $q$ is in $\rArea{J_1} \cup \rArea{J_2}$.
Since $q$ was arbitrary, we deduce that (\ref{Eqn:R2covered}) holds.

 Furthermore, by definition of $J_2$ and Lemma \ref{Lemma:Place2}, we have
\begin{equation} \label{Eqn:J2bound}
 |J_2| \leq |\mathcal{N}| \leq \frac{d (k^d - | \rArea{J'}|)}{r}. 
\end{equation}

Finally we select repeats from $J'$ that cover $R_3 \cap \rArea{J'}$. Let $E$ be a face of $F_k$ of dimension $d_0 \in [\ell+1,d]$, and let $\cC = \{ E \cap S_2 : (S_1,S_2) \in J'\}$. Consider $E$ as a $k$-hypercube in $\Z^{d_0}$.  Let $p$ be in $\interior_n(E)$. In $\Z^{d_0}$, there $n^{d_0}/3^{d_0}$ translates of $F_n$ whose centers are within $n/6$ of $p$. Since $n^{d_0}/3^{d_0} \geq n^{\ell+1}/3^d > |W_n(u)|$, the patterns in $u$ appearing at these translates cannot all be distinct, meaning that there is a repeat. Thus, for each $p$ in $\interior_n(E)$, there exists a repeat $(S_1,S_2)$ in $J'$ such that $\rho(p,c) \leq n/6$, where $c$ is the center of $S_2$. Applying Lemma \ref{Lemma:CoverInterior}, we conclude that there exists $J_3(E) \subset J'$ such that
\begin{equation} \label{Eqn:J3covers}
 \interior_n(E) \subset \bigcup_{(S_1,S_2) \in J_3(E)} S_2 \cap E,
\end{equation}
and
\begin{equation} \label{Eqn:J3Facebound}
 |J_3(E)| \leq \biggl(\frac{2k}{n}\biggr)^{d_0}.
\end{equation}
By (\ref{Eqn:J3covers}), we have that 
\begin{equation} \label{Eqn:J3covers2}
 T_n(E) \subset \bigcup_{(S_1,S_2) \in J_3(E)} S_2.
\end{equation}
Let $J_3 = \cup_E J_3(E)$, where the union runs over all faces $E$ with dimension in $[\ell+1,d]$. Then by (\ref{Eqn:J3covers2}) and (\ref{Eqn:R3decomp}), we have
\begin{equation} \label{Eqn:R3covered}
 R_3 \cap \rArea{J'} = R_3 \cap \rArea{J_3}.
\end{equation}
By (\ref{Eqn:J3Facebound}), we have
\begin{equation} \label{Eqn:J3bound}
 |J_3| \leq \sum_{d_0=\ell+1}^d c_{d,d_0} \biggl( \frac{2 k}{n} \biggr)^{d_0}.
\end{equation}

Finally, we set $J = J_1 \cup J_2 \cup J_3$. By (\ref{Eqn:R1covered}), (\ref{Eqn:R2covered}), and (\ref{Eqn:R3covered}), we have that
\begin{equation} \label{Eqn:JandJprime}
 \rArea{J'} = \rArea{J}.
\end{equation}
Since $J'$ was an $n$-repeat cover of $u$ and $J$ is a subset of $J'$ satisfying (\ref{Eqn:JandJprime}), we have that $J$ is an $n$-repeat cover of $u$. Furthermore, by (\ref{Eqn:J1bound}), (\ref{Eqn:J2bound}),  (\ref{Eqn:J3bound}), and (\ref{Eqn:JandJprime}), we have that
\begin{align*}
 |J| & \leq |J_1| + |J_2| + |J_3| \\
 & \leq 2 c_{d,\ell} \frac{k^{\ell} r^{d-\ell}}{n} + \frac{d (k^d - |\rArea{J}|)}{r} + \sum_{d_0=\ell+1}^d c_{d,d_0} \biggl( \frac{2 k}{n} \biggr)^{d_0}.
\end{align*}
\end{proof}

We conclude this section with the following lemma, which quantifies the asymptotic efficiency that we can guarantee for repeat covers. This lemma plays a crucial role in obtaining the lower bound on the probability of emptiness in Section \ref{Sect:EmptinessProof}. The proof of the lemma involves a direct application of Lemma \ref{Lemma:ThreePieces} and some calculations. In Remark \ref{Rmk:ExampleSequences}, we provide examples of sequences $\{f(n)\}$ and $\{r_n\}$ that satisfy the hypotheses of the lemma.

\begin{lemma} \label{Lemma:FinalCoveringLemma}
 Suppose $\{f(n)\}$ and $\{r_n\}$ are sequences such that
\begin{enumerate}
\item $f(n) \to \infty$;
\item  $\log(n) / r_n \to 0$;
\item $f(n) r_n = o\bigl( (n/ \log n)^{1/d} \bigr)$.
\end{enumerate}
Let $k = k(n) = n f(n)$. For any $\delta >0$, there exists $n_0$ such that if $n \geq n_0$ and $u \in \cA^{F_k}$ with $|W_n(u)| = j \in [n/5^d, (k-n+1)^d]$, then there exists a repeat cover $J$ for $u$ such that
\begin{equation*}
|J| \leq \delta \frac{j}{\log n}.
\end{equation*}
\end{lemma}
\begin{proof}
Fix $f(n)$ and $r_n$ as above, and let $k = k(n) = n f(n)$.
Let $\delta >0$.
By hypotheses (2) and (3), we may write $r_n = \frac{\log n}{g_1(n)}$ and $f(n) r_n = g_2(n) (n/ \log n)^{1/d}$, where $g_1(n)$ and $g_2(n)$ tend to $0$ as $n$ tends to infinity. For $u$ in $\cA^{F_k}$ with $|W_n(u)|= j \geq n^d / 3^d$, we apply Lemma \ref{Lemma:GeneralCoding} to the hypercube $F_k$ (considered as a face of $F_k$ of dimension $d$) and conclude that there exists a repeat cover $J$ for $u$ such that
\begin{align*}
 |J| & \leq 2|F_k|/n = 2 k^d /n.
\end{align*}
Then since $f(n)^d \leq g_2(n)^d n / \log n$ and $j \geq n^d / 3^d$, we have
\begin{align*}
|J| & \leq 2 k^d/n \\ &  = 2 n^{d-1} f(n)^d \\ & \leq 2 n^{d-1} g_2(n)^d n/\log n  \\ & = 2 g_2(n)^d n^d / \log n \\ & \leq 2 \cdot 3^d g_2(n)^d \frac{j}{\log n}.
\end{align*}
Since $g_2(n)$ tends to $0$, there exists $n_1$ such that if $n \geq n_1$, then $2 \cdot 3^d g_2(n)^d \leq \delta$. Then for $n \geq n_1$, we have that $|J| \leq \delta j / \log n$.

Now suppose $n^{\ell} / 5^d \leq j < n^{\ell+1}/3^d$ for some $\ell \in [1,d-1]$. Then by Lemma \ref{Lemma:ThreePieces}, for each $u$ in $\cA^{F_k}$ with $|W_n(u)|=j$, there exists a repeat cover $J$ such that
\begin{align} \label{Eqn:ThreePieces}
|J| & \leq 2 c_{d,\ell} \frac{k^{\ell} r_n^{d-\ell}}{n} + \frac{d(k^d-|\rArea{J}|)}{r_n} + \sum_{d_0=\ell+1}^d c_{d,d_0} \biggl( \frac{2 k}{n} \biggr)^{d_0}.
\end{align}
Fix $\epsilon >0$. By Lemma \ref{Lemma:Nuggets} and the fact that $n/k  = 1/f(n)$ tends to $0$ (hypothesis (1)), there exists $n_2$ such that if $n \geq n_2$ then $k^d - |\rArea{J}| \leq (1+4dn/k) j \leq (1+\epsilon) j$. 
Applying this inequality in (\ref{Eqn:ThreePieces}) gives that for $n \geq n_2$, we have
\begin{align} \label{Eqn:Carrboro}
|J| & \leq 2 c_{d,\ell} \frac{k^{\ell} r_n^{d-\ell}}{n} + \frac{d(1+\epsilon)j}{r_n} + \sum_{d_0=\ell+1}^d c_{d,d_0} \biggl( \frac{2 k}{n} \biggr)^{d_0}.
\end{align}
Note that $c_{d,d_0} \leq 2^d$, and then we have
\begin{equation} \label{Eqn:Hillsborough}
\sum_{d_0=\ell+1}^d c_{d,d_0} \biggl( \frac{2 k}{n} \biggr)^{d_0} \leq d4^d k^d / n^d.
\end{equation}
Then for $n \geq n_2$, by (\ref{Eqn:Carrboro}) and (\ref{Eqn:Hillsborough}) we obtain that 
\begin{align*}
|J| & \leq 2 c_{d,\ell} \frac{k^{\ell} r_n^{d-\ell}}{n} + d(1+\epsilon)\frac{j}{r_n} + d4^d k^d / n^d \\ 
& \leq 2 c_{d,\ell} n^{\ell-1} g_2(n)^d n/\log n + d(1+\epsilon) \frac{g_1(n)j}{\log n} + d 4^d g_2(n)^d n / \log n \\
& \leq 2 c_{d,\ell} n^{\ell}  g_2(n)^d / \log n  + d(1+\epsilon) \frac{g_1(n) j}{\log n} + d 4^d g_2(n)^d n /\log n \\
& \leq (2c_{d,\ell}+d 4^d) \frac{g_2(n)^d}{\log n} n^{\ell} + d(1+\epsilon) \frac{g_1(n)j}{\log n} \\
& \leq (2c_{d,\ell}+d 4^d) \frac{g_2(n)^d}{\log n} 5^d j + d(1+\epsilon) \frac{g_1(n)j}{\log n} \\
& \leq \biggl(5^d(2 c_{d,\ell}+d 4^d) g_2(n)^d  + d(1+\epsilon) g_1(n) \biggr) \frac{j}{\log n}.
\end{align*}
As $g_1(n)$ and $g_2(n)$ tend to $0$ as $n$ tends to infinity, we see that there exists $n_0 \geq \max(n_1,n_2)$ such that for $n \geq n_0$, we have
\begin{equation*}
|J| \leq \delta \frac{j}{\log n}.
\end{equation*}
\end{proof}

\begin{rmk} \label{Rmk:ExampleSequences}
 As an example of sequences $\{f(n)\}$ and $\{r_n\}$ that satisfy the hypotheses of Lemma \ref{Lemma:FinalCoveringLemma}, one may take $f(n) = r_n = n^{\tau}$ for any $\tau$ satisfying $0 < \tau < \frac{1}{2d}$. In this case, one obtains $k = k(n) =  n^{1+\tau}$, $g_1(n) = \log(n)/ n^{\tau}$, and $g_2(n) = \log(n)^{1/d} n^{2 \tau - \frac{1}{d}}$. With this parametrization, the choice of $\tau$ optimizing the upper bound in the proof of Lemma~\ref{Lemma:FinalCoveringLemma} is given by $\tau = \frac{1}{2d+1}$, which yields the estimate $|J| \leq C j / n^{\frac{1}{2d+1}}$ for some constant $C$.
\end{rmk}

\subsection{Proofs of Theorems \ref{Thm:EmptinessIntro} and \ref{Thm:ExoticBehavior}} \label{Sect:EmptinessProof}

In this section, we present proofs of Theorems \ref{Thm:EmptinessIntro} and \ref{Thm:ExoticBehavior}. As mentioned previously, the proof of Theorem \ref{Thm:EmptinessIntro} involves finding both upper and lower bounds on the probability of emptiness. The upper bound is given by Proposition \ref{Prop:UB}. The lower bound requires additional work, and Lemma \ref{Lemma:FinalCoveringLemma} plays a crucial role in that regard. 

\vspace{2mm}

\begin{PfofEmptinessThm}{}
 Let $X = \mathcal{A}^{\Z^d}$. Note that the radius of convergence of $\zeta_X(t)$ is $|\cA|^{-1}$. 
Also note that in the trivial case $\alpha = 0$, we have  $\mathbb{P}_{n,0}(\mathcal{E}_n) = 1 = \zeta_{X}(0)^{-1}$, so the conclusion of the theorem holds in this case.


First consider the case $\alpha \geq |\cA|^{-1}$. By inclusion, we have $\PP(\mathcal{E}_n) \leq \PP( \Per(X_{\omega}) = \varnothing )$. Combining this inequality with Proposition \ref{Prop:UB} and letting $n$ tend to infinity, we obtain the conclusion of the theorem for $\alpha \geq |\cA|^{-1}$ (since $\zeta_X(t)$ diverges to infinity for $t \geq |\cA|^{-1}$).

For the rest of the proof, we assume $\alpha < |\cA|^{-1}$. 
Observe that for any $k \geq n$, we have that
\begin{align} \label{Eqn:Inclusion}
\PP(\mathcal{E}_n) \geq \PP \bigl(W_k(X_{\omega} \bigr) = \varnothing),
\end{align}
by inclusion.
For each $u$ in $\cA^{F_k}$, let $F(u)$ be the event that $u$ is forbidden (\textit{i.e.}, $W_n(u) \cap \mathcal{F}(\omega) \neq \varnothing$). Then by (\ref{Eqn:Inclusion}) we have
\begin{equation}  \label{Eqn:Intersection}
\PP(\mathcal{E}_n) \geq \PP \bigl(W_k(X_{\omega}) = \varnothing \bigr) = \PP \Biggl( \bigcap_{u \in \cA^{F_k}} F(u) \Biggr).
\end{equation}
Let $P_j$ be the set of finite orbits $\gamma$ in $X$ such that $|\gamma | = j$.
To each $\gamma$ in $P_j$ with $j \leq n/2$, we associate a pattern $u_{\gamma}$ in $\cA^{F_k}$ such that $W_n(u_{\gamma}) = W_n(\gamma)$ (and therefore $|W_n(u_{\gamma})| = |\gamma|$). Let $S_0 = \{u_{\gamma} : |\gamma|\leq n/2\}$, and let
\begin{equation*}
 S_1 = \biggl\{ u \in \cA^{F_k} : \forall \gamma \in \bigcup_{j \leq n/2} P_j, \, W_n(\gamma) \setminus W_n(u) \neq \varnothing \biggr\}.
\end{equation*}
Let $S = S_0 \sqcup S_1$. Note that if $u \in \cA^{F_k} \setminus S_1$, then there exists a finite orbit $\gamma$ such that $|\gamma|\leq n/2$ and $W_n(\gamma) \subset W_n(u)$, and therefore $F(u_{\gamma}) \subset F(u)$. Hence, we have
\begin{equation} \label{Eqn:SufficientIntersection}
\bigcap_{u \in \cA^{F_k}} F(u) = \bigcap_{u \in S} F(u).
\end{equation}

Since each $F(u)$ is a monotone decreasing event (meaning that if $\omega \notin F(u)$ and $\omega \leq \tau$ coordinate-wise, then $\tau \notin F(u)$), then by (\ref{Eqn:Intersection}), (\ref{Eqn:SufficientIntersection}), and the FKG inequality (see \cite{Grimmett} for a proof), we get
\begin{align}
\begin{split} \label{Eqn:FKG}
\PP(\mathcal{E}_n) & \geq \PP \Biggl( \bigcap_{u \in \cA^{F_k}} F(u) \Biggr) \\ & = \PP \Biggl( \bigcap_{u \in S} F(u) \Biggr) \\ & \geq \prod_{u \in S}  \PP(  F(u) ). 
\end{split}
\end{align}

Using (\ref{Eqn:FKG}) and then re-writing, we obtain
\begin{align}
\begin{split} \label{Eqn:Rewrite}
\PP(\mathcal{E}_n) & \geq \prod_{u \in S}  \PP(  F(u) ) \\
 & = \prod_{u \in S} (1-\alpha^{|W_n(u)|}) \\
 & = \prod_{u \in S_0} (1-\alpha^{|W_n(u_{\gamma})|}) \prod_{u \in S_1} (1-\alpha^{|W_n(u)|}) \\
 & = \prod_{|\gamma|\leq n/2 } (1-\alpha^{|\gamma|}) \prod_{u \in S_1} (1-\alpha^{|W_n(u)|}) .
\end{split}
\end{align}

By Lemma \ref{Lemma:Smallj}, if $u$ is in $S_1$, then $|W_n(u)| > n/2$. From  (\ref{Eqn:Rewrite}) and this fact, we see that
\begin{align}
\begin{split} \label{Eqn:UseSmalljLemma}
 \PP(\mathcal{E}_n) & \geq \prod_{|\gamma|\leq n/2 } (1-\alpha^{|\gamma|}) \prod_{u \in S_1} (1-\alpha^{|W_n(u)|}) \\
 & \geq \prod_{|\gamma|\leq n/2 } (1-\alpha^{|\gamma|}) \prod_{|W_n(u)| > n/2} (1-\alpha^{|W_n(u)|}).
\end{split}
\end{align}
Recall the notation (from (\ref{Eqn:Nnkj}))
\begin{equation*}
N_{n,k}^j = \{u \in \cA^{F_k} : |W_n(u)|=j\}.
\end{equation*}
Then (\ref{Eqn:UseSmalljLemma}) gives
\begin{equation} \label{Eqn:LCD}
\PP(\mathcal{E}_n) \geq \prod_{|\gamma|\leq n/2 } (1-\alpha^{|\gamma|}) \prod_{j=\lfloor n/2 \rfloor +1}^{(k-n+1)^d} (1-\alpha^j)^{|N_{n,k}^j|}.
\end{equation}

Let us now show that there exist $C>0$ and $0< \beta <1$ such that for large enough $n$, we have
\begin{equation} \label{Eqn:ExpEmptinessLB}
  \prod_{j=n/2}^{(k-n+1)^d} (1-\alpha^j)^{|N_{n,k}^j|} \geq \exp\Bigl( - C \beta^{n} \Bigr).
\end{equation}

%

 As $\alpha |\cA| < 1$, there exist $\epsilon>0$ and $\delta>0$ such that $\alpha |\cA|^{1+\epsilon} < 1$ and $8 \delta  d < -\log(\alpha |\cA|^{1+\epsilon})$. Also, let $\{f(n)\}$ and $\{r_n\}$ be sequences satisfying the hypotheses of Lemma \ref{Lemma:FinalCoveringLemma} (they may be taken as in Remark \ref{Rmk:ExampleSequences}), and let $k = k(n) = n f(n)$.
By Lemma \ref{Lemma:FinalCoveringLemma}, there exists $n_0$ such that if $n \geq n_0$ and $u  \in N_{n,k}^j$ with $n/5^d \leq j \leq (k-n+1)^d$, 
 then there exists a repeat cover $J$ of $u$ such that
\begin{equation*} 
 |J| \leq \delta \frac{j}{\log n}.
\end{equation*}
For the moment, fix such a $j$. 
To each $u$ in $N_{n,k}^j$, let $J_u$ be a repeat cover of $u$ such that $|J_u| \leq \delta j / \log n$, and let $w_u = u|_{F_k \setminus \rArea{J_u}} \in \cA^{F_k \setminus \rArea{J_u}}$. Let
\begin{equation*}
 U = \{ (J,w) : J \subset \cC_n(F_k) \times \cC_n(F_k), \, w \in \cA^{F_k \setminus \rArea{J}} \}.
\end{equation*}
Define a map $\phi : N_{n,k}^j \to U$ by $\phi(u) = (J_u,w_u)$. By Lemma \ref{Lemma:DefineAWord}, $\phi$ is injective, so that $|N_{n,k}^j| = |\phi(N_{n,k}^j)|$. Also, let $\mathcal{P}(S)$ denote the power set of a set $S$, and define a map $\pi : U \to \mathcal{P}(\cC_n(F_k) \times \cC_n(F_k))$ by $\pi(J,w) = J \subset \cC_n(F_k) \times \cC_n(F_k)$. By construction, 
\begin{align}
 \begin{split} \label{Eqn:ByConstruction}
\pi \circ \phi(N_{n,k}^j) & = \{ J_u : u \in N_{n,k}^j \} \\
 &   \subset \{ J \subset \cC_n(F_k) \times \cC_n(F_k) : |J| \leq \delta j / \log n\}  .
 \end{split}
\end{align}
Using  (\ref{Eqn:ByConstruction}) and $|\cC_n(F_k)| \leq k^d$, we see that 
\begin{align} 
\begin{split} \label{Eqn:SizePiPhi}
 | \pi \circ \phi(N_{n,k}^j)| & \leq |\{ J \subset \cC_n(F_k) \times \cC_n(F_k) : |J| \leq \delta j / \log n\}| \\ & \leq (k^{2d})^{ \delta j / \log n+1}.
\end{split}
\end{align}
Also, for each $J \subset \cC_n(F_k) \times \cC_n(F_k)$,
\begin{equation} \label{Eqn:PreImageJ}
 |\pi^{-1}(J)| \leq |\cA|^{|F_k \setminus \rArea{J}|} = |\cA|^{k^d - |\rArea{J}|}.
\end{equation}
By Lemma \ref{Lemma:Nuggets} and the fact that $n/k = 1/f(n) \to 0$, there exists $n_1 \geq n_0$ such that if $n \geq n_1$ and $u \in N_{n,k}^j$, then 
\begin{equation} \label{Eqn:UseNuggets}
 k^d - |\rArea{J_u}| \leq (1+4dn/k)j \leq (1+\epsilon) j.
\end{equation}
 
By (\ref{Eqn:SizePiPhi}), (\ref{Eqn:PreImageJ}), and (\ref{Eqn:UseNuggets}), we obtain that for $n \geq n_1$ and $n/5^d \leq j \leq (k-n+1)^d$, 
\begin{align}
\begin{split} \label{Eqn:UseThree}
|N_{n,k}^j| & = |\phi(N_{n,k}^j)| \\
 & \leq |\pi \circ \phi(N_{n,k}^j)| \cdot \max \{ |\pi^{-1}(J)| : J \in \pi \circ \phi(N_{n,k}^j) \} \\
 & \leq k^{2d (\delta j +\log n) / \log n} \cdot \max \{|\cA|^{k^d - |\rArea{J_u}|} : u \in N_{n,k}^j \} \\
 & \leq k^{2d  (\delta j + \log n)  / \log n} \cdot |\cA|^{(1+\epsilon) j}.
\end{split}
\end{align} 
Since $k = nf(n)$ and $\{f(n)\}$ satisfies hypothesis (3) from Lemma \ref{Lemma:FinalCoveringLemma}, there exists $n_2 \geq n_1$ such that if $n \geq n_2$, then $k \leq n^2$. Then by (\ref{Eqn:UseThree}),  for $n \geq n_2$, we have
\begin{align}
\begin{split} \label{Eqn:UseChoiceDelta}
\alpha^j |N_{k,n}^j|  & \leq \alpha^j |\cA|^{(1+\epsilon)j} k^{ 2 d (\delta j + \log n) / \log n} \\
& = \exp \biggl( j \log (\alpha |\cA|^{1+\epsilon}) + 2 d ( \delta j + \log n) \log (k) / \log n \biggr) \\
& \leq \exp \biggl(  j \Bigl( \log (\alpha |\cA|^{1+\epsilon}) + 4 d (\delta + \log (n)/j) \Bigr) \biggr).
\end{split}
\end{align}
Since $j \geq n/5^d$ (so that $\log(n) /j$ tends to $0$), our choice of $\delta$ implies that there exists $n_3 \geq n_2$ such that if $n \geq n_3$ then $4d (\delta + \log(n) /j) \leq - \log( \alpha |\cA|^{1+\epsilon})/2$. Hence for $n \geq n_3$, (\ref{Eqn:UseChoiceDelta}) yields
\begin{align} \label{Eqn:UseChoiceDelta2}
\alpha^j |N_{k,n}^j| \leq \exp \biggl( j  \log ( \alpha |\cA|^{1+\epsilon})/ 2 \biggr).
\end{align}
Letting $\beta_0 = (\alpha |\cA|^{1+\epsilon})^{1/2} <1$, we see from (\ref{Eqn:UseChoiceDelta2}) that 
\begin{align} \label{Eqn:InTermsBeta}
\sum_{j = n/2}^{(k-n+1)^d} \alpha^j |N_{k,n}^j| & \leq \sum_{j = n/2}^{(k-n+1)^d} \beta_0^j  \leq \beta_0^{n/2} \frac{1}{1- \beta_0}.
\end{align}
By calculus, there exist $n_4 \geq n_3$ and $C_0>0$ such that for $j \geq n_4/2$,
\begin{equation} \label{Eqn:Calculus}
- \log (1 - \alpha^{j}) \leq C_0 \alpha^{j}.
\end{equation}
Then for $n \geq n_4$, by (\ref{Eqn:InTermsBeta}) and (\ref{Eqn:Calculus}), we have that
\begin{align*}
 \prod_{j = \lfloor n/2 \rfloor +1}^{(k-n+1)^d} (1-\alpha^j)^{|N_{n,k}^j|} & = \exp\Biggl( \sum_{j = n/2}^{(k-n+1)^d} |N_{n,k}^j| \log(1-\alpha^j) \Biggr) \\
 & \geq \exp\Biggl( - C_0  \sum_{j = n/2}^{(k-n+1)^d} \alpha^j |N_{n,k}^j|  \Biggr) \\
 & \geq \exp\Biggl( - C_0 \beta_0^{n/2} \frac{1}{1- \beta_0}  \Biggr).
\end{align*}
Taking $C = C_0/(1-\beta_0)$ and $\beta = \beta_0^{1/2}$ establishes (\ref{Eqn:ExpEmptinessLB}).

Note that (\ref{Eqn:UseSmalljLemma}), (\ref{Eqn:LCD}) and (\ref{Eqn:ExpEmptinessLB}) together imply that for large enough $n$,
\begin{equation} \label{Eqn:ExpLB}
 \PP(\mathcal{E}_n) \geq \prod_{|\gamma|\leq n/2 } (1-\alpha^{|\gamma|}) \exp\Bigl( - C \beta^{n} \Bigr).
\end{equation}
Then (\ref{Eqn:ExpLB}) shows that for large enough $n$, we have
\begin{equation} \label{Eqn:ExpLB2}
 \PP(\mathcal{E}_n) \geq \zeta_X(\alpha)^{-1} \exp \Bigl(-C\beta^n \bigr).
\end{equation}

By (\ref{Eqn:ExpLB2}) and Proposition \ref{Prop:UB} (see (\ref{Eqn:ExpUBsubcrit})), we have that for $\alpha \in (0,\, |\cA|^{-1})$, there exist $C_1,C_2>0$ and $\beta_1,\beta_2,\in (0,1)$ such that for large enough $n$,
\begin{equation} \label{Eqn:Slanted}
 1 - C_1 \beta_1^n \leq \frac{\PP(\mathcal{E}_n)}{\zeta_X(\alpha)^{-1}} \leq 1 + C_2 \beta_2^n.
\end{equation}
Using (\ref{Eqn:Slanted}) and calculus, we have that
 there exist $C_3>0$ and $\beta_3 \in (0,1)$ such that for large enough $n$, 
\begin{align*}
 |\PP(\mathcal{E}_n) - \zeta_X(\alpha)^{-1}| & \leq \zeta_X(\alpha)^{-1}  C_3 \beta_3^n,
\end{align*}
which concludes the proof of the theorem.
\end{PfofEmptinessThm}

%

\vspace{2mm}

We are now in a position to prove Theorem \ref{Thm:ExoticBehavior}. The proof of Theorem \ref{Thm:ExoticBehavior} follows easily from Proposition \ref{Prop:UB} and an estimate obtained in the proof of Theorem \ref{Thm:EmptinessIntro}.

\vspace{2mm}

\begin{PfofExoticThm}{} Let $\cA$, $d$, and $\mathcal{G}_n$ be as the statement of the theorem.
By inclusion, we have 
\begin{equation} \label{Eqn:SimpleInclusion}
\PP(\mathcal{G}_n) \leq \PP(\Per(X_{\omega})=\varnothing). 
\end{equation}
For $\alpha > |\cA|^{-1}$, the probability that $X_{\omega}$ has no finite orbits tends to $0$ at least exponentially in $n$ by Proposition \ref{Prop:UB} (see (\ref{Eqn:ExpUBsupercrit})), which establishes the conclusion of the theorem for $\alpha > |\cA|^{-1}$. For $\alpha = |\cA|^{-1}$, the combination of (\ref{Eqn:SimpleInclusion}) with Proposition \ref{Prop:UB} yields the desired conclusion.

Now suppose $\alpha < |\cA|^{-1}$. By (\ref{Eqn:ExpEmptinessLB}) and calculus, there exist $C_4,C_5>0$ and $\beta_4,\beta_5 \in (0, 1)$ such that for large enough $n$,
\begin{equation} \label{Eqn:PerLB}
 \PP(\mathcal{E}_n) \geq \zeta_X(\alpha)^{-1} \exp( -C_4 \beta_4^n ) \geq \zeta_X(\alpha)^{-1}( 1- C_5 \beta_5^n).
\end{equation}
By (\ref{Eqn:ExpUBsubcrit}), (\ref{Eqn:ExpEmptinessLB}), and (\ref{Eqn:PerLB}), we have that there exist $C_6,C_7>0$ and $\beta_6,\beta_7 \in (0,1)$ such that for large enough $n$,
\begin{align} \label{Eqn:Wowee}
\begin{split}
 \PP(\mathcal{G}_n) & = \PP( \{ \Per(X_{\omega}) = \varnothing \} \setminus  \mathcal{E}_n ) \\
 & =  \PP(  \Per(X_{\omega}) = \varnothing  ) - \PP(\mathcal{E}_n) \\
 & \leq  \zeta_X(\alpha)^{-1}( 1 + C_6 \beta_6^n) - \zeta_X(\alpha)^{-1}( 1- C_7 \beta_7^n) \\
 & \leq \zeta_X(\alpha)^{-1} C_8 \beta_8^n,
\end{split}
\end{align}
where $C_8 = 2 \max(C_6,C_7)$ and $\beta_8 = \max(\beta_6,\beta_7)$. By (\ref{Eqn:Wowee}), we have completed the proof.
\end{PfofExoticThm}

\section{Entropy} \label{Sect:Entropy}

In this section we investigate the entropy of random $\Z^d$-SFTs. The proof of Theorem \ref{Thm:EntropyIntro}, presented in Section \ref{Sect:EntropyProof}, involves two second moment arguments, one providing an upper bound on entropy and the other providing a lower bound on entropy. To prepare for this proof, we prove several lemmas in Section \ref{Sect:EntropyLemmas}. These lemmas estimate the asymptotic behavior of the first and second moments of several random variables used to bound entropy. One random variable counts the number of patterns on $F_k$ that avoid the forbidden $F_n$-patterns (see (\ref{Eqn:CountWords})), and the other random variable counts the average number of allowed $F_k$-patterns that can fill in a fixed periodic boundary pattern (see (\ref{Eqn:CountPerWords})).

\subsection{Entropy lemmas} \label{Sect:EntropyLemmas}

Throughout this section we will assume that $k >n$. Also, as standing notation, we set $\ell =  k - n +1$. For each $u$ in $\cA^{F_k}$, let $\xi_u$ be the random variable that is $1$ when $u$ is allowed (\textit{i.e.}, when $W_n(u) \cap \cF(\omega) = \varnothing$) and $0$ otherwise. 

\subsubsection{Lemmas for upper bound on entropy} \label{Sect:EntropyLemmasUB}
Let $\phi_{n,k}$ be the number of allowed $F_k$-patterns:
\begin{equation} \label{Eqn:CountWords}
 \phi_{n,k} = \sum_{u \in \cA^{F_k}} \xi_u.
\end{equation}
The following two lemmas, which concern the expectation and variance of $\phi_{n,k}$, will be used to give an upper bound on the limiting distribution of entropy. We begin by describing the asymptotic behavior of the expectation of $\phi_{n,k}$.

\begin{lemma} \label{Lemma:EPhi}
For any $k >n$, it holds that
\begin{equation} \label{Eqn:ExpPhinkRawLB}
 \EE( \phi_{n,k} ) \geq \alpha^{\ell^d} |\cA^{F_k}| = \alpha^{\ell^d} |\cA|^{k^d}.
\end{equation}
Furthermore, if $\alpha > |\cA|^{-1}$ and $k = k(n)$ satisfies $n/k \to 0$ and $\log (k)/n \to 0$, then
\begin{equation*}
 \lim_n \EE(\phi_{n,k})^{1/k^d} = \alpha |\cA|.
\end{equation*}
\end{lemma}
\begin{proof}
Recall $N_{n,k}^j = \{ u \in \cA^{F_k} : |W_n(u)| = j\},$ and note that
\begin{align} \label{Eqn:ExpPhink}
 \EE( \phi_{n,k} ) = \sum_{u \in \cA^{F_k} } \EE( \xi_u ) = \sum_{u \in \cA^{F_k}} \alpha^{|W_n(u)|} = \sum_{j=1}^{\ell^d} \alpha^j \bigl|N_{n,k}^j \bigr|.
\end{align}
Then by (\ref{Eqn:ExpPhink}),
\begin{equation*} 
 \EE( \phi_{n,k} ) \geq \alpha^{\ell^d} \bigl|\cA^{F_k} \bigr| = \alpha^{\ell^d} |\cA|^{k^d},
\end{equation*}
which verifies (\ref{Eqn:ExpPhinkRawLB}).

Now assume that  $\alpha > |\cA|^{-1}$ (so that $\alpha |\cA| >1$) and $k = k(n)$ satisfies $n/k \to 0$ and $ \log (k) / n \to 0$. From (\ref{Eqn:ExpPhinkRawLB}), we see that
\begin{equation} \label{Eqn:ExpPhinkLB}
 \liminf_n \Bigl( \EE(\phi_{n,k}) \Bigr)^{1/k^d} \geq \liminf_n \alpha^{(\ell/k)^d} |\cA| = \alpha |\cA|,
\end{equation}
since $\ell/k = 1 - n/k +1/k$, which tends to one as $n$ tends to infinity.

Let us now bound the cardinality of $N_{n,k}^j$ (for each $j \in [1,\ell^d]$) from above. Let $\epsilon >0$, and consider $j$ in $[1,\ell^d]$. By Lemma \ref{Lemma:GeneralCoding}, for each pattern $u$ in $N_{n,k}^j$, there exists a repeat cover $J$ of $u$ such that $|J| \leq 2 k^d /n$. By Lemma \ref{Lemma:Nuggets} and the fact that $n/k \to 0$, for $n$ large enough, any repeat cover $J$ of a pattern $u$ in $N_{n,k}^j$ satisfies $|F_k \setminus \rArea{J}| \leq (1+4dn/k)j \leq (1+\epsilon)j$. Assume now that we have $n$ large enough for these inequalities to hold. By Lemma \ref{Lemma:DefineAWord}, a pattern $u$ in $N_{n,k}^j$ is determined by any repeat cover $J$ for $u$ and the pattern $w = u|_{F_k \setminus \rArea{J}}$. Thus, bounding $|N_{n,k}^j|$ by the number of pairs $(J,w)$, where $J \subset \cC_n(F_k) \times \cC_n(F_k)$ with $|J| \leq 2 k^d /n$ and $w \in \cA^{F_k \setminus \rArea{J}}$, we obtain
\begin{equation} \label{Eqn:BoundNnkj}
 |N_{n,k}^j| \leq |\cA|^{(1+\epsilon)j} (k^{2d})^{2 k^d/n+1}.
\end{equation}
Here $(k^{2d})^{2 k^d/n+1}$ is an upper bound on the number of subsets $J$ of $\cC_n(F_k) \times \cC_n(F_k)$ such that $|J| \leq 2 k^d/n$, and $|\cA|^{(1+\epsilon)j}$ is an upper bound on the number of patterns in $\cA^{F_k \setminus \rArea{J}}$ where $|F_k \setminus \rArea{J}| \leq (1+\epsilon)j$.



Let $p_1(k) = k^{4d+1}$. Then by (\ref{Eqn:BoundNnkj}), for any $\epsilon >0$ and large enough $n$,  we have that
\begin{align}
\begin{split} \label{Eqn:ExpPhinkOne}
 \EE( \phi_{n,k} ) &  = \sum_{j=1}^{\ell^d} \alpha^j |N_{n,k}^j| \\
 & \leq \alpha^{\ell^d} |\cA|^{k^d} + \sum_{j = 1}^{\ell^d-1} \alpha^j |\cA|^{(1+\epsilon)j} p(k)^{k^d/n} \\
 & \leq \alpha^{\ell^d} |\cA|^{k^d} + p_1(k)^{k^d/n} \sum_{j=1}^{\ell^d-1} (\alpha |\cA|^{1+\epsilon})^j \\
 & \leq \alpha^{\ell^d} |\cA|^{k^d} + p_1(k)^{k^d/n} ( \alpha |\cA|^{1+\epsilon})^{\ell^d} (\alpha |\cA|^{1+\epsilon} - 1)^{-1} \\
 & \leq 2 \max \biggl( \alpha^{\ell^d} |\cA|^{k^d}, \; p_1(k)^{k^d/n} ( \alpha |\cA|^{1+\epsilon})^{\ell^d} (\alpha |\cA|^{1+\epsilon} - 1)^{-1} \biggr).
\end{split}
\end{align}
By (\ref{Eqn:ExpPhinkOne}) and our assumptions on $\alpha$ (\textit{i.e.}, $\alpha |\cA| > 1$) and on $k$ (\textit{i.e.}, $n/k \to 0$ and $\log (k)/n  \to 0$), we obtain
\begin{align*}
 \begin{split}
  \limsup_n \Bigl( \EE( \phi_{n,k} ) \Bigr)^{1/k^d} \leq \alpha |\cA|^{1+\epsilon}.
 \end{split}
\end{align*}
As $\epsilon >0$ was arbitrary, we see that
\begin{equation} \label{Eqn:ExpPhinkUB}
 \limsup_n \Bigl( \EE( \phi_{n,k} ) \Bigr)^{1/k^d} \leq \alpha |\cA|.
\end{equation}
From (\ref{Eqn:ExpPhinkLB}) and (\ref{Eqn:ExpPhinkUB}), we have that
\begin{equation*}
 \lim_n \Bigl( \EE( \phi_{n,k} ) \Bigr)^{1/k^d} = \alpha |\cA|.
\end{equation*}
\end{proof}

In the following lemma, we show that the variance of $\phi_{n,k}$ is small compared to the square of its expectation.

\begin{lemma} \label{Lemma:VarPhi}
 Suppose $\alpha |\cA| > 1$ and $k = k(n) = n f(n)$ with $f(n) \to \infty$ and $f(n) = o( (n/ \log n)^{1/d} )$. Then there exist $K_1>0$ and $\rho_1>0$ such that for large enough $n$,
\begin{equation*}
 \frac{\Var(\phi_{n,k})}{\EE(\phi_{n,k})^2} \leq K_1 \exp( - \rho_1 n^d ).
\end{equation*}
\end{lemma}
\begin{proof}
 Suppose that $\alpha$ and $k=k(n)$ are as above.
We introduce the notation
\begin{equation*}
D_{n,k}^j = \{ (u,v) \in \cA^{F_k} \times \cA^{F_k} : W_n(u) \cap W_n(v) \neq \varnothing, \, |W_n(u) \cup W_n(v) | = j \}.
\end{equation*}

Observe that the covariance of $\xi_u$ and $\xi_v$ is
\begin{align}
\begin{split} \label{Eqn:Zowee}
 \EE\biggl( \bigl( \xi_u - \EE(\xi_u) \bigr) \bigl( \xi_v - \EE(\xi_v) \bigr) \biggr) & = \EE\bigl( \xi_u \xi_v \bigr) - \EE(\xi_u) \EE(\xi_v) \\
 & = \alpha^{|W_n(u) \cup W_n(v)|} - \alpha^{|W_n(u)|+|W_n(v)|} \\
 & = \alpha^{|W_n(u) \cup W_n(v)|} \bigl( 1- \alpha^{|W_n(u) \cap W_n(v)|} \bigr).
\end{split}
\end{align}
Then, since the variance of a sum is the sum of the covariances, we have
\begin{align}
\begin{split} \label{Eqn:VarPhinkRawUB}
 \Var(\phi_{n,k}) & = \sum_{u,v \in \cA^{F_k}} \alpha^{|W_n(u) \cup W_n(v)|} \Bigl( 1 - \alpha^{|W_n(u) \cap W_n(v)|} \Bigr) \\
 & \leq \sum_{\substack{u,v \in \cA^{F_k} \\ W_n(u) \cap W_n(v) \neq \varnothing}} \alpha^{|W_n(u) \cup W_n(v)|} \\
 & = \sum_{j = 1}^{2\ell^d -1} \alpha^j |D_{n,k}^j|.
\end{split}
\end{align}
Then by (\ref{Eqn:ExpPhinkRawLB}) and (\ref{Eqn:VarPhinkRawUB}), we have that
\begin{align}
 \begin{split} \label{Eqn:VarPhinkFirst}
  \frac{\Var(\phi_{n,k})}{\EE(\phi_{n,k})^2} & \leq \frac{\sum_{j = 1}^{2\ell^d -1} \alpha^j |D_{n,k}^j|}{ \alpha^{2 \ell^d} |\cA|^{2k^d}} 
 \end{split}
\end{align}

Let $(u,v)$ be in $D_{n,k}^j$. Let $F'_k$ be a disjoint copy of $F_k$; for concreteness, we take $F'_k = F_k + k\mathbf{1}$. For notational convenience, we assume that $v \in \cA^{F'_k}$, so that we may think of $(u,v)$ as an element of $\cA^{F_k \sqcup F'_k}$ and use the terminology of repeats as in Section \ref{Sect:Repeats}. 
Let $r$ be the total number of $n$-repeats in $(u,v)$. Let $p$ be in $F_{\ell} \cup (F_{\ell} +k\mathbf{1})$, and let $S = F_n + p - \1$. Note that either there exists a unique repeat $(S_1,S_2)$ for $(u,v)$ such that $S = S_2$ or else $S$ is the lexicographically minimal appearance of the pattern $(u,v)|_S$ in $(u,v)$. Thus, we have that
\begin{equation*}
 2 \ell^d = j + r,
\end{equation*}
and hence $r = 2 \ell^d - j$. 

Now let $V$ be the repeat region $\rArea{J}$, where $J$ is any repeat cover for $(u,v)$ (recall that the set $\rArea{J} $ does not depend on the choice of $J$).
Let $(S^*_1,S^*_2)$ be the repeat in $(u,v)$ such that $S^*_2$ is lexicographically minimal. For any $n$-cube $S$ in $\cC_n(F_k \sqcup F'_k)$, let $M(S)$ denote the lexicographically maximal element of $S$. Define a map $\Phi$ from the $n$-repeats in $(u,v)$ into the power set of $V$ as follows:
\begin{equation*}
 \Phi(S_1,S_2) = \left\{ \begin{array}{ll}
                          S_2, & \text{ if } (S_1,S_2) = (S^*_1,S^*_2) \\
                          \{M(S_2)\}, & \text{ otherwise}.
                         \end{array}
                   \right.
\end{equation*}
Note that if $(S_1,S_2)$ and $(S_3,S_4)$ are repeats in $(u,v)$, then $\Phi(S_1,S_2) \cap \Phi(S_3,S_4) = \varnothing$. Thus, we have shown that $|V| \geq n^d + (r-1)$, and therefore
\begin{equation*}
 |V| \geq n^d + (r-1) = 2 \ell^d - j + n^d -1.
\end{equation*}
Then
\begin{equation} \label{Eqn:Mercer}
 |(F_k \sqcup F'_k) \setminus V| = 2 k^d - |V| \leq 2 k^d - 2 \ell^d - n^d + j + 1.
\end{equation}
By Lemma \ref{Lemma:GeneralCoding}, there exists a repeat cover $J$ of $(u,v)$ such that $|J| \leq 4 k^d /n$. 
By Lemma \ref{Lemma:DefineAWord}, any element $(u,v)$ in $D_{n,k}^j$ is uniquely determined by a repeat cover $J$ for $(u,v)$ and the pattern $w = (u,v)|_{(F_k \sqcup F'_k) \setminus \rArea{J}}$. Thus, we may bound the cardinality of $D_{n,k}^j$ by the number of pairs $(J,w)$, where $J \subset \cC_n(F_k \sqcup F'_k) \times \cC_n(F_k \sqcup F'_k)$ satisfies $|J| \leq 4 k^d/n$ and $w$ is an element of $\cA^{(F_k \sqcup F'_k) \setminus \rArea{J}}$ (and we know that $|(F_k \sqcup F'_k) \setminus \rArea{J}|$ may be bounded as in (\ref{Eqn:Mercer})). In this way, we obtain  that 
\begin{align} 
\begin{split}\label{Eqn:DnkBound}
|D_{n,k}^j| & \leq (4k^{2d})^{4 k^d/n+1}  |\cA|^{2 k^d - 2 \ell^d - n^d + j + 1} \\
& \leq  p_2(k)^{k^d/n}  |\cA|^{2 k^d - 2 \ell^d - n^d + j + 1},
\end{split}
\end{align}
where $4k^{2d}$ is an upper bound on the number of pairs of $n$-cubes contained in $F_k \sqcup F'_k$ and $p_2(k) = (4 k^{2d})^5$.

By 
(\ref{Eqn:VarPhinkFirst}) and (\ref{Eqn:DnkBound}), we have
\begin{align}
\begin{split} \label{Eqn:VarPhinkSecond}
 \frac{\Var(\phi_{n,k})}{\EE(\phi_{n,k})^2} & \leq \frac{\sum_{j = 1}^{2\ell^d -1} \alpha^j |D_{n,k}^j|}{ \alpha^{2 \ell^d} |\cA|^{2k^d}} \\
 & \leq \frac{ p_2(k)^{k^d/n} |\cA|^{2 k^d - 2 \ell^d - n^d + 1} \sum_{j=1}^{2\ell^d-1} (\alpha |\cA|)^{j} }{ \alpha^{2 \ell^d} |\cA|^{2k^d}} \\
 & \leq \frac{ p_2(k)^{k^d/n} |\cA|^{2 k^d - 2 \ell^d - n^d + 1} (\alpha |\cA|)^{2 \ell^d} (\alpha |\cA| - 1)^{-1} }{ \alpha^{2 \ell^d} |\cA|^{2k^d}} \\
 & = \frac{ p_2(k)^{k^d/n} |\cA| (\alpha |\cA|-1)^{-1} }{ |\cA|^{n^d} }.
 \end{split}
\end{align}
By hypothesis, we have $k = n f(n)$, and we may write $f(n) = g_1(n) (n/\log n)^{1/d}$ with $g_1(n) \to 0$. Note that for large enough $n$, we must have $k \leq n^2$ and $p_2(k) \leq k^{10d+1}$. Then by (\ref{Eqn:VarPhinkSecond}), for $n$ large enough, we have
\begin{align}
\begin{split} \label{Eqn:VarPhinkThird}
 \frac{\Var(\phi_{n,k})}{\EE(\phi_{n,k})^2} & \leq \frac{ p_2(k)^{k^d/n} |\cA| (\alpha |\cA|-1)^{-1} }{ |\cA|^{n^d} } \\
  & = |\cA| (\alpha |\cA| - 1)^{-1} \exp \biggl( \frac{k^d}{n} \log p_2(k) - n^d \log |\cA| \biggr) \\
  & \leq |\cA| (\alpha |\cA| - 1)^{-1} \exp \biggl( (10d+1) \frac{n^d g_1(n)^d}{\log n} \log k - n^d \log |\cA| \biggr) \\
  & \leq |\cA| (\alpha |\cA| - 1)^{-1} \exp \biggl( 2(10d+1) n^d g_1(n)^d - n^d \log |\cA| \biggr) \\
  & =  |\cA| (\alpha |\cA| - 1)^{-1} \exp \biggl( n^d (2(10d+1) g_1(n)^d - \log |\cA|) \biggr).
 \end{split}
\end{align}

By (\ref{Eqn:VarPhinkThird}) and the fact that $g_1(n) \to 0$, we see that there exist $K_1>0$ and $\rho_1>0$ such that for large enough $n$,
\begin{equation*}
 \frac{\Var(\phi_{n,k})}{\EE(\phi_{n,k})^2} \leq K_1 \exp( - \rho_1 n^d ).
\end{equation*}
\end{proof}

\subsubsection{Lemmas for lower bound on entropy} \label{Sect:EntropyLemmasLB}

The next two lemmas will be used to give a lower bound on the limiting distribution of entropy. In order to state these lemmas, we need some additional notation. 
Let $V_{n,k}$ be the (nonrandom) set of periodic patterns on $\partial_n F_k$ (with period $\ell$ in each cardinal direction):
\begin{equation*}
V_{n,k} = \{ u \in \cA^{\partial_n F_k} :  u_{t+\ell e_i} = u_t \text{ whenever } t, t+\ell e_i \in \partial_n F_k \}.
\end{equation*}
Also, let $P_{n,k}$ be the set of $F_k$-patterns with periodic $n$-boundaries:
\begin{equation*}
P_{n,k} = \{ u \in \cA^{F_k} : u|_{\partial_n F_k} \in V_{n,k}\}.
\end{equation*}
Note for future reference that $|P_{n,k}|=|\cA|^{\ell^d}$.
Let $\psi_{n,k}$ be the average number of ways of filling in a periodic boundary (\textit{i.e.}, a boundary from $V_{n,k}$):
\begin{equation} \label{Eqn:CountPerWords}
\psi_{n,k} = \frac{1}{|V_{n,k}|} \sum_{b \in P_{n,k}} \xi_{b}.
\end{equation}
Note that $\psi_{n,k}$ is a random quantity. 
The following lemma shows that $\psi_{n,k}$ may be used to give a lower bound on the entropy of the random $\Z^d$-SFT.

\begin{lemma} \label{Lemma:LowerBdLemma}
 For $k>n$ and any $\omega$ in $\Omega_n$, it holds that
\begin{equation*}
 \frac{1}{k^d} \log \psi_{n,k}(\omega) \leq h_{\text{per}}(X_{\omega}) \leq h(X_{\omega}).
\end{equation*}
\end{lemma}
\begin{proof}
 Let $k>n$, and let $\omega$ be in $\Omega_n$. 

For a periodic boundary $b$ in $V_{n,k}$, let  $W(b)$ be the set of patterns $u$ in $W_k(X_{\omega})$ such that $u|_{\partial_n F_k} = b$. Since the forbidden patterns have shape $F_n$ and the patterns in $V_{n,k}$ have thickness $n$, we have that
\begin{equation*}
 \psi_{n,k}(\omega) = \frac{1}{|V_{n,k}|} \sum_{b \in V_{n,k}} |W(b)|.
\end{equation*}
Since the average of a finite set of real numbers is less than or equal to its maximum, there exists $b$ in $V_{n,k}$ such that $|W(b)| \geq \psi_{n,k}(\omega)$, and we fix such a pattern $b$. 

By definition, $X_{\omega}$ is a $\Z^d$-SFT, and the forbidden patterns that define $X_{\omega}$ all have shape $F_n$. Thus, if $\{w_p\}_{p \in \Z^d}$ is any collection of patterns in $W(b)$ indexed by $\Z^d$, then there is a point $x$ in $X_{\omega}$ such that $x|_{F_k + \ell p} = w_p$ for all $p$ in $\Z^d$. In particular, if $\{w_p\}_{p \in \Z^d}$ is totally periodic, then the corresponding point $x$ has a finite orbit. Therefore
\begin{equation*}
 \frac{1}{\ell^d} \log |W(b)| \leq h_{\text{per}}(X_{\omega}) \leq h(X_{\omega}).
\end{equation*}
Then, using that $|W(b)| \geq \psi_{n,k}(\omega)$ and $\ell \leq k$, we see that
\begin{equation*}
 \frac{1}{k^d} \log \psi_{n,k}(\omega) \leq  \frac{1}{\ell^d} \log |W(b)| \leq h_{\text{per}}(X_{\omega}) \leq h(X_{\omega}),
\end{equation*}
as desired.
\end{proof}

In the following lemma, we describe the asymptotic behavior of the expectation of $\psi_{n,k}$.

\begin{lemma} \label{Lemma:EPsi}
 For any $k >n$, it holds that
\begin{equation} \label{Eqn:EphiRawLB}
 \EE( \psi_{n,k} ) \geq |V_{n,k}|^{-1} \alpha^{\ell^d} |P_{n,k}| = |V_{n,k}|^{-1} \alpha^{\ell^d} |\cA|^{\ell^d}.
\end{equation}
Further, if $\alpha |\cA| > 1$ and $k = k(n)$ satisfies $n/k \to 0$ and $\log (k) / n \to 0$, then
\begin{equation*}
 \lim_n \EE(\psi_{n,k})^{1/k^d} = \alpha |\cA|.
\end{equation*}
\end{lemma}
\begin{proof}
 Define
\begin{equation*}
Q_{n,k}^j = \{ u \in P_{n,k} : |W_n(u)|=j\},
\end{equation*}
and note that
\begin{align} 
\begin{split}\label{Eqn:RewritePsi}
|V_{n,k}| \, \EE(\psi_{n,k}) & =  \sum_{b \in P_{n,k}} \EE(\xi_b)  \\  & = \sum_{b \in P_{n,k}} \alpha^{|W_n(b)|} \\ & = \sum_{j=1}^{\ell^d} \alpha^j |Q_{n,k}^j|.
\end{split}
\end{align}
Recall that $|P_{n,k}| = |\cA|^{\ell^d}$. Then using (\ref{Eqn:RewritePsi}), we obtain
\begin{align*} 
\EE(\psi_{n,k}) &  = |V_{n,k}|^{-1} \sum_{j=1}^{\ell^d} \alpha^j |Q_{n,k}^j| \\ & \geq |V_{n,k}|^{-1} \alpha^{\ell^d} |P_{n,k}| \\ & = |V_{n,k}|^{-1} \alpha^{\ell^d} |\cA|^{\ell^d},
\end{align*}
which verifies (\ref{Eqn:EphiRawLB}).

Now suppose that $\alpha > |\cA|^{-1}$ (so that $\alpha |\cA|>1$) and $k  = k(n)$ satisfies $n/k \to 0$ and $ \log (k) /n \to 0$.
Since $P_{n,k} \subset \cA^{F_k}$ and $|V_{n,k}|\geq 1$, we have $\psi_{n,k} \leq \phi_{n,k}$. Then by Lemma \ref{Lemma:EPhi}, we have that
\begin{equation} \label{Eqn:EphiUB}
 \limsup_n \EE(\psi_{n,k})^{1/k^d} \leq \limsup_n \EE(\phi_{n,k})^{1/k^d} = \alpha |\cA|.
\end{equation}

For the sake of notation, define $B = \log_{|\cA|} |V_{n,k}|$, and note
that $B \leq |\partial_n F_k| \leq K(d) n k^{d-1}$ for some constant $K(d)$ that depends only on $d$, which implies that $B/k^d \to 0$. Then by (\ref{Eqn:EphiRawLB}) and the fact that $n/k \to 0$, we have that
\begin{equation} \label{Eqn:EphiLB}
\liminf_n \EE(\psi_{n,k})^{1/k^d} \geq \liminf_n |\cA|^{-B/k^d} (\alpha |\cA|)^{(\ell/k)^d} = \alpha |\cA|.
\end{equation}
By (\ref{Eqn:EphiUB}) and (\ref{Eqn:EphiLB}), we obtain
\begin{equation*}
 \lim_n \EE(\psi_{n,k})^{1/k^d} = \alpha |\cA|.
\end{equation*}
\end{proof}

The following lemma shows that the variance of $\psi_{n,k}$ is small compared to the square of its expectation.

\begin{lemma} \label{Lemma:VarPsi}
 Suppose $\alpha |\cA| > 1$ and $k = k(n) = n f(n)$ with $f(n) \to \infty$ and $f(n) = o( (n/ \log n)^{1/d} )$. Then there exist $K_2>0$ and $\rho_2>0$ such that for large enough $n$,
\begin{equation*}
 \frac{\Var(\psi_{n,k})}{\EE(\psi_{n,k})^2} \leq K_2 \exp( - \rho_2 n^d ).
\end{equation*}
\end{lemma}
\begin{proof}
Define
\begin{equation*}
\hat{S}_{n,k}^j = \{(u,v) \in P_{n,k} \times P_{n,k} : W_n(u) \cap W_n(v) \neq \varnothing, \, |W_n(u) \cup W_n(v)| = j\},
\end{equation*}
and
\begin{equation*}
S_{n,k}^j   =  \bigcup_{i=1}^j \hat{S}_{n,k}^i.
\end{equation*}
As in (\ref{Eqn:Zowee}) and (\ref{Eqn:VarPhinkRawUB}), we have
\begin{align}
\begin{split} \label{Eqn:VarPsiFirst}
\Var(\psi_{n,k}) & = |V_{n,k}|^{-2} \sum_{u,v \in P_{n,k}} \alpha^{|W_n(u) \cup W_n(v)|} \Bigl( 1 - \alpha^{|W_n(u) \cap W_n(v)|} \Bigr) \\
 &  \leq |V_{n,k}|^{-2} \sum_{\substack{u,v \in P_{n,k} \\ W_n(u) \cap W_n(v) \neq \varnothing}} \alpha^{|W_n(u) \cup W_n(v)|} \\
 & = |V_{n,k}|^{-2}\sum_{j = 1}^{2\ell^d-1} \alpha^j |\hat{S}_{n,k}^j|. 
\end{split}
\end{align}

Let $b = b(n) = 2\ell^d - n^d$. Then by Lemma \ref{Lemma:EPsi} and (\ref{Eqn:VarPsiFirst}), we have
\begin{align}
\begin{split} \label{Eqn:VarPsiSecond}
\frac{\Var(\psi_{n,k})}{\EE(\psi_{n,k})^2} & \leq \frac{\sum_{j = 1}^{2\ell^d-1} \alpha^j |\hat{S}_{n,k}^j|}{ (\alpha |\cA|)^{2\ell^d} } \\
 & \leq \frac{\sum_{j = 1}^{b-1} \alpha^j |\hat{S}_{n,k}^j| + \sum_{j = b}^{2\ell^d-1} \alpha^j |\hat{S}_{n,k}^j|}{ (\alpha |\cA|)^{2\ell^d} } \\
 & \leq \frac{\sum_{j = 1}^{b-1} \alpha^j |\hat{S}_{n,k}^j| + \alpha^b  \sum_{j = b}^{2\ell^d-1} |\hat{S}_{n,k}^j|}{ (\alpha |\cA|)^{2\ell^d} } \\
& \leq \frac{\sum_{j = 1}^{b-1} \alpha^j |\hat{S}_{n,k}^j| + \alpha^b  |S_{n,k}^{2\ell^d-1}|}{ (\alpha |\cA|)^{2 \ell^d} }
\end{split}
\end{align}
We proceed by finding upper bounds for $|\hat{S}_{n,k}^j|$ and $|S_{n,k}^{2\ell^d-1}|$.

First, we seek an upper bound for $|\hat{S}_{n,k}^j|$. As in the proof of Lemma \ref{Lemma:VarPhi}, we consider pairs $(u,v)$ in $P_{n,k} \times P_{n,k}$ as elements of $\cA^{F_k \sqcup F'_k}$, where we take $F'_k = F_k + k \1$ for concreteness. For such a pair $(u,v)$, let $r$ be the total number of repeats in $(u,v)$. Then, as before, $2 \ell^d = j + r$, and so $r = 2\ell^d - j$. Let $V$ be the repeat region, \textit{i.e.}, $V = \rArea{J}$ for any repeat cover $J$ of $(u,v)$ (recall that the set $\rArea{J}$ is independent of the choice of $J$). The map $(S_1,S_2) \mapsto m(S_2)$ is an injection of the set of repeats in $(u,v)$ into the set $V \cap ( F_{\ell} \sqcup (F_{\ell}+k \1))$, and therefore
\begin{equation*}
 |V \cap \bigl(F_{\ell} \sqcup (F_{\ell}+k\mathbf{1}) \bigr) | \geq r = 2 \ell^d - j.
\end{equation*}
Hence,
\begin{align}
\begin{split} \label{Eqn:GoBears}
 |\bigl(F_{\ell} \sqcup (F_{\ell}+k\mathbf{1}) \bigr) \setminus V| & = 2\ell^d - | V \cap ( F_{\ell} \sqcup (F_{\ell}+k\mathbf{1})) | \\ & \leq 2 \ell^d - (2\ell^d - j) = j.
\end{split}
\end{align}
By Lemma \ref{Lemma:GeneralCoding} (viewing the full $d$-cube $F_k$ as a face of dimension $d$), for each $(u,v)$ in $\hat{S}_{n,k}^j$, there exists a repeat cover $J$ of $(u,v)$ such that $|J| \leq 4 k^d/n$. Furthermore, 
each element $(u,v)$ of $\hat{S}_{n,k}^j$ is uniquely determined by a repeat cover $J$ and the pattern $w=(u,v)|_{\bigl(F_{\ell} \sqcup (F_{\ell}+k\mathbf{1})\bigr) \setminus \rArea{J}}$ (by Lemma \ref{Lemma:DefineAWord} and the fact that $u$ and $v$ have periodic boundary). Thus, bounding $|\hat{S}_{n,k}^j|$ by the number of pairs $(J,w)$ such that $J \subset \cC_n(F_k \sqcup F'_k) \times \cC_n(F_k \sqcup F'_k)$ with $|J| \leq 4 k^d/n$ and $w \in \cA^{\bigl(F_{\ell} \sqcup (F_{\ell}+k\mathbf{1})\bigr) \setminus \rArea{J}}$ and using (\ref{Eqn:GoBears}), we see that
\begin{equation} \label{Eqn:SnkEst1}
 |\hat{S}_{n,k}^j | \leq (4k^{2d})^{4 k^d/n+1} |\cA|^j.
\end{equation}
Here we have used that $4k^{2d}$ is an upper bound on the number of pairs of hypercubes in $F_k \sqcup F'_k$. As before, let $p_2(k) = (4k^{2d})^5$.

Let us now obtain an upper bound on the cardinality of $S_{n,k}^{2\ell^d-1}$. Here we will use the fact that for any pair $(u,v)$ in $S_{n,k}^{2\ell^d-1}$, we have that $W_n(u) \cap W_n(v) \neq \varnothing$.
Let $(u,v)$ be in $S_{n,k}^{2\ell^d - 1}$. Due to the periodicity of the boundary of $v$, the pattern $v|_{F_{\ell}+q}$ uniquely determines $v$ for each $q$ such that $F_{\ell}+q \subset F'_k$. Choose the lexicographically minimal $n$-cube $S_2^*$ in $\cC_n(F'_k)$ such that $v|_{S_2^*} \in W_n(u)$ (which exists since $W_n(u) \cap W_n(v) \neq \varnothing$), and then choose the lexicographically minimal $n$-cube in $S_1^*$ in $F_k$ such that $v|_{S_1^*} = v|_{S_2^*}$. Now choose the lexicographically minimal point $q$ in $F'_k$ such that $S_2^* \subset F_{\ell}+q \subset F'_k$ (and note that $q$ is uniquely determined by $S_2^*$). Let $\varphi$ map $(u,v)$ to $(S_1^*,S_2^*,u,v|_{(F_{\ell}+q)\setminus S_2^*})$. Then $\varphi$ is injective, and therefore
\begin{equation} \label{Eqn:SnkEst2}
|S_{n,k}^{2 \ell^d -1}| = |\varphi(S_{n,k}^{2\ell^d - 1})| \leq  k^{2d} |\cA|^{\ell^d} |\cA|^{\ell^d - n^d} =  k^{2d} |\cA|^{2\ell^d - n^d}. 
\end{equation}
Here we have used that $k^{2d}$ is an upper bound on the cardinality of the possible pairs $(S_1,S_2)$ appearing in the first two coordinates of the image of $\varphi$. Note that $k^{2d} \leq p_2(k)= (4k^{2d})^5$.

By (\ref{Eqn:VarPsiSecond}), (\ref{Eqn:SnkEst1}) and (\ref{Eqn:SnkEst2}), we have
\begin{align}
 \begin{split} \label{Eqn:VarPsinkLast}
  \frac{\Var(\psi_{n,k})}{\EE(\psi_{n,k})^2} & \leq \frac{\sum_{j = 1}^{b-1} \alpha^j |\hat{S}_{n,k}^j| + \alpha^b  |S_{n,k}^{2\ell^d-1}|}{ (\alpha |\cA|)^{2 \ell^d} } \\
 & \leq \frac{p_2(k)^{k^d/n} \sum_{j = 1}^{b-1} \alpha^j |\cA|^j + \alpha^b  p_2(k) |\cA|^{2\ell^d - n^d}} { (\alpha |\cA|)^{2 \ell^d} }  \\
 & \leq \frac{ (\alpha |\cA| - 1)^{-1} p_2(k)^{k^d/n} (\alpha |\cA|)^{b}}{ (\alpha |\cA|)^{2 \ell^d}}  + \frac{ (\alpha |\cA|)^{2 \ell^d - n^d}  p_2(k) } { (\alpha |\cA|)^{2 \ell^d} }  \\
 & \leq \frac{ (\alpha |\cA| - 1)^{-1} p_2(k)^{k^d/n} }{ (\alpha |\cA|)^{n^d}}  + \frac{  p_2(k) } { (\alpha |\cA|)^{n^d} } .
 \end{split}
\end{align}

By hypothesis, we may write $k = n f(n)$, where $f(n) = o((n/\log n)^{1/d})$. Using this hypothesis, the fact that $\alpha |\cA| >1$ and (\ref{Eqn:VarPsinkLast}), we see that there exist $K_2>0$ and $\rho_2>0$ such that for large enough $n$,
\begin{equation*}
 \frac{\Var(\psi_{n,k})}{\EE(\phi_{n,k})^2} \leq K_2 \exp( - \rho_2 n^d ).
\end{equation*}
\end{proof}

\subsection{Proof of Theorems \ref{Thm:EntropyIntro} and \ref{Thm:PeriodicEntropyIntro}} \label{Sect:EntropyProof}

Here we present a unified proof of Theorems \ref{Thm:EntropyIntro} and \ref{Thm:PeriodicEntropyIntro}. The proof essentially breaks into two parts, the upper bound on entropy and the lower bound on the periodic entropy. In each case we use a second moment argument, relying on Chebyshev's inequality.  The upper bound is a consequence of Lemmas \ref{Lemma:EPhi} and \ref{Lemma:VarPhi}, and the lower bound is a consequence of Lemmas \ref{Lemma:EPsi} and \ref{Lemma:VarPsi}. 

\vspace{2mm}

\begin{PfofEntropyThm}{}
Choose $k = k(n) = n f(n)$ with $f(n) \to \infty$ and $f(n) = o((n/\log n)^{1/d})$.
By subadditivity, for any $k >n$, we have that
\begin{equation} \label{Eqn:hUB}
 h(X_{\omega}) \leq \frac{1}{k^d} \log |W_k(X_{\omega})| = \frac{1}{k^d} \log \phi_{n,k}.
\end{equation}
Let $\beta_n(\omega) = \exp( h(X_{\omega}))$ and $\eta_n(\omega) = \exp( h_{\text{per}}(X_{\omega}) )$. Then (\ref{Eqn:hUB}) may be rewritten as
\begin{equation} \label{Eqn:betaUB}
 \beta_n^{k^d} \leq |W_k(X_{\omega})| = \phi_{n,k}.
\end{equation}
On the other hand, 
by Lemma \ref{Lemma:LowerBdLemma}, we have that
\begin{equation*}
 \frac{1}{k^d} \log \psi_{n,k} \leq h_{\text{per}}(X_{\omega}) = \log \eta_n, 
\end{equation*}
and therefore
\begin{equation} \label{Eqn:betaLB}
 \psi_{n,k} \leq \eta_n^{k^d}.
\end{equation}

Suppose $\alpha > |\cA|^{-1}$ (so that $\alpha |\cA| > 1$), and let $\epsilon >0$ be such that $\alpha |\cA| - \epsilon >1$. Then
\begin{align}
 \begin{split} \label{Eqn:ProbHfar}
  \PP\Bigl( | \beta_n - \alpha |\cA| | \geq \epsilon & \text{ or } | \eta_n - \alpha |\cA| | \geq \epsilon  \Bigr) \\
 & \leq \PP\Bigl( \phi_{n,k} \geq (\alpha |\cA| +\epsilon)^{k^d} \Bigr) +  \PP\Bigl( \psi_{n,k} \leq (\alpha |\cA| - \epsilon)^{k^d} \Bigr),
 \end{split}
\end{align}
where the inequality follows from (\ref{Eqn:betaUB}), (\ref{Eqn:betaLB}), and $\eta_n \leq \beta_n$ by inclusion.
Note that
\begin{align*}
 \begin{split}
  \PP\Bigl( \phi_{n,k} & \geq (\alpha |\cA| +\epsilon)^{k^d} \Bigr) \\
 & = \PP\Bigl( \phi_{n,k} - \EE(\phi_{n,k}) \geq (\alpha |\cA| +\epsilon)^{k^d} - \EE(\phi_{n,k}) \Bigr) \\
 & = \PP\Biggl( \phi_{n,k} - \EE(\phi_{n,k}) \geq \EE(\phi_{n,k}) \biggl( \biggl(\frac{\alpha |\cA| +\epsilon}{\EE(\phi_{n,k})^{1/k^d}}\biggr)^{k^d} - 1 \biggr) \Biggr). 
 \end{split}
\end{align*}
Let $d_1 = (\Var(\phi_{n,k}))^{1/2} / \EE(\phi_{n,k})$. Then by Chebyshev's inequality, 
\begin{align*}
\begin{split}
 \PP\Bigl( & \phi_{n,k}  \geq (\alpha |\cA| +\epsilon)^{k^d} \Bigr) \\
 & = \PP \Biggl( \phi_{n,k} - \EE(\phi_{n,k}) \geq (\Var(\phi_{n,k}))^{1/2} \frac{1}{d_1} \biggl( \biggl(\frac{\alpha |\cA| +\epsilon}{\EE(\phi_{n,k})^{1/k^d}}\biggr)^{k^d} - 1 \biggr) \Biggr) \\
 & \leq \Biggl( \frac{d_1}{\bigl((\alpha |\cA| +\epsilon)/\EE(\phi_{n,k})^{1/k^d}\bigr)^{k^d} - 1 } \Biggr)^2.
\end{split}
\end{align*}
Since $\EE(\phi_{n,k})^{1/k^d}$ tends to $\alpha |\cA|$ (Lemma \ref{Lemma:EPhi}) and $d_1^2 \leq K_1 \exp( - \rho_1 n^d)$ for large enough $n$ (Lemma \ref{Lemma:VarPhi}), we obtain that there exist $K_3>0$ and $\rho_3 >0$ such that for large enough $n$,
\begin{equation} \label{Eqn:ProbHtooBig}
 \PP\Bigl( \phi_{n,k} \geq (\alpha |\cA| +\epsilon)^{k^d} \Bigr) \leq K_3 \exp( - \rho_3 n^d).
\end{equation}

Similarly, letting $d_2 = (\Var(\psi_{n,k}))^{1/2} / \EE(\psi_{n,k})$ and using Chebyshev's inequality gives
\begin{align*}
 \PP\Bigl( & \psi_{n,k} \leq (\alpha |\cA| - \epsilon)^{k^d} \Bigr) \\
 & = \PP \Biggl( \psi_{n,k} - \EE(\psi_{n,k}) \geq (\Var(\psi_{n,k}))^{1/2} \frac{1}{d_2} \biggl( \biggl(\frac{\alpha |\cA| -\epsilon}{\EE(\psi_{n,k})^{1/k^d}}\biggr)^{k^d} - 1 \biggr) \Biggr) \\
 & \leq \Biggl( \frac{d_2}{\bigl((\alpha |\cA| - \epsilon)/\EE(\psi_{n,k})^{1/k^d}\bigr)^{k^d} - 1 } \Biggr)^2.
\end{align*}
Since $\EE(\psi_{n,k})^{1/k^d}$ tends to $\alpha |\cA|$ (Lemma \ref{Lemma:EPsi}) and $d_2^2 \leq K_2 \exp( - \rho_2 n^d)$ for large enough $n$ (Lemma \ref{Lemma:VarPsi}), we obtain that there exist $K_4 >0$ and $\rho_4>0$ such that for large enough $n$,
\begin{equation} \label{Eqn:ProbHtooSmall}
 \PP\Bigl( \psi_{n,k} \leq (\alpha |\cA| -\epsilon)^{k^d} \Bigr) \leq K_4 \exp( - \rho_4 n^d).
\end{equation}
Combining (\ref{Eqn:ProbHfar}), (\ref{Eqn:ProbHtooBig}) and (\ref{Eqn:ProbHtooSmall}), we have that there exist $K_5$ and $\rho_5$ such that for large enough $n$,
\begin{equation*}
 \PP\Bigl( | \beta_n - \alpha |\cA| | \geq \epsilon \text{ or } |\eta_n - \alpha |\cA| | \geq \epsilon  \Bigr) \leq K_5 \exp( - \rho_5 n^d),
\end{equation*}
which is equivalent to the conclusions of the theorems for $\alpha > |\cA|^{-1}$ by the continuity of the logarithm.

Now suppose that $\alpha \leq |\cA|^{-1}$. Let $\epsilon > 0$. Choose $\alpha' > |\cA|^{-1}$ such that $\log(\alpha' |\cA|) = \epsilon/2$. Since entropy is a monotone increasing random variable (\textit{i.e.}, if $\omega \leq \omega'$ then $h(X_{\omega}) \leq h(X_{\omega'})$) and $\alpha < \alpha'$, then by \cite[Theorem 2.1]{Grimmett}, we have that
\begin{equation*}
 \PP( h(X_\omega) \geq \epsilon ) \leq \P_{n,\alpha'}( h(X_{\omega}) \geq \epsilon).
\end{equation*}
Also, since $\log(\alpha' |\cA|) = \epsilon/2$, we have
\begin{equation*}
 \P_{n,\alpha'} \Bigl( h(X_{\omega}) \geq \epsilon \Bigr) = \P_{n,\alpha'} \Bigl( h(X_{\omega}) - \log(\alpha' |\cA|) \geq \epsilon/2 \Bigr).
\end{equation*}
Since $\alpha' > |\cA|^{-1}$, we have already shown that there exist $K >0$ and $\rho >0$ such that for large enough $n$,
\begin{equation*}
 \P_{n,\alpha'}\Bigl( h(X_{\omega}) - \log(\alpha' |\cA|) \geq \epsilon/2 \Bigr) \leq K \exp ( - \rho n^d ).
\end{equation*}
By combining the previous three displays and using that $h_{\text{per}}(X_{\omega}) \leq h(X_{\omega})$, we obtain the desired conclusions.
\end{PfofEntropyThm}



\section{Discussion} \label{Sect:Discussion}

Let us close with some general remarks and open questions regarding the behavior of random $\Z^d$-SFTs.
\begin{rmk}
 There is a more general setting for $\Z^d$-SFTs than the one considered in this work. Suppose $X$ is a $\Z^d$-SFT. Then one may obtain a probability distribution on the $\Z^d$-SFTs contained in $X$ by randomly forbidding patterns from $W_n(X)$ with some probability $\alpha$. Hence, one may ask about the likely properties of random $\Z^d$-SFTs contained inside an ambient $\Z^d$-SFT $X$. In \cite{McGoff}, this more general setting was studied for $\Z$-SFTs, and results analogous to Theorems \ref{Thm:EmptinessIntro} and \ref{Thm:EntropyIntro} were shown to hold whenever $X$ is an irreducible $\Z$-SFT. In this work, we only allow the ambient shift $X$ to be a full shift, as full shifts seem to provide the only cleanly defined class of $\Z^d$-SFTs for $d \geq 2$ that possess all of the properties required for our proofs. Nonetheless, it would be interesting to understand the behavior of random $\Z^d$-SFTs inside of other ambient shifts.
\end{rmk}

\begin{rmk}
It has been quite difficult to prove $\Z^d$-SFT versions of many fundamental theorems about $\Z$-SFTs (for example, consider factor theorems \cite{Boyle1983,BPS,Marcus1979}, embedding theorems \cite{Krieger1982,Lightwood2003,Lightwood2004}, and uniqueness of measure of maximal entropy theorems \cite{BurtonSteif,Parry1964}). 
The difficulty in extending such theorems is often caused by the strange or pathological behavior that can occur in some $\Z^d$-SFTs, which either removes hope for the $\Z^d$ result entirely or forces stringent hypotheses to absolutely rule out the "bad" examples. 

However, the results of this paper suggest that "typical" $\Z^d$-SFTs may avoid these pathological behaviors. Therefore, it may be possible to prove versions of $\Z$-theorems for "typical" $\Z^d$-SFTs. 
In particular, there may be some $\Z$-SFT theorems which hold, not for all $\Z^d$-SFTs, but for sets of $\Z^d$-SFTs that have probability tending to one as $n$ approaches infinity (for certain values of $\alpha$). 
For example, one may ask whether  ``typical'' $\Z^d$-SFTs have a unique measure of maximal entropy.
Previous work  \cite[Theorem 1.4]{McGoff} implies that for $\alpha$ close to one, a random $\Z$-SFT has a unique measure of maximal entropy with probability tending to one as $n$ tends to infinity. 
\end{rmk}

\begin{rmk}
 As mentioned in the introduction, the class of $\Z^d$-SFTs exhibits strikingly different behavior in the two cases $d=1$ and $d>1$. However, our results suggest that these differences may not appear for ``typical'' systems. Indeed, the main results presented in this work give a precise sense in which typical $\Z^d$-SFTs behave similarly with respect to emptiness, entropy, and periodic points, regardless of $d$. Thus, there remains an interesting open question, which we formulate as follows. Does there exist a property $\mathcal{Q}$ of SFTs, a probability $\alpha \in [0,1]$, an alphabet $\cA$, and natural numbers $d_1 \neq d_2$ such that $\inf_n \PP(\mathcal{Q}) >0$ in dimension $d_1$ and $\lim_n \PP(\mathcal{Q}) = 0$ in dimension $d_2$? 
\end{rmk}

\section*{Acknowledgements}
The authors thank Mike Boyle for valuable comments regarding early versions of the manuscript.

\bibliographystyle{abbrv}
\bibliography{RandomZdSFTsRefs}

\end{document}